\documentclass[reqno]{amsart}
\usepackage{amsmath}%
\usepackage{amsthm}
\usepackage{amsfonts}%
\usepackage{amssymb}%
\usepackage{graphicx, caption}
\usepackage{comment}
\usepackage{tikz-cd}
\usetikzlibrary{cd}
\usepackage{bm}

\usepackage{esint} 

\usepackage{xcolor}

\usepackage{mathrsfs}

\usepackage{stmaryrd} 
\usepackage{mathtools} 
\usepackage{todonotes} 
\usepackage{bbm} 
\usepackage[hidelinks]{hyperref}

\DeclareFontFamily{OT1}{pzc}{}
\DeclareFontShape{OT1}{pzc}{m}{it}{<-> s * [1.10] pzcmi7t}{}
\DeclareMathAlphabet{\mathpzc}{OT1}{pzc}{m}{it}

\mathtoolsset{showonlyrefs}

\newtheorem{theorem}{Theorem}[section]

\newtheorem{question}[theorem]{Question}

\newtheorem{corollary}[theorem]{Corollary}

\newtheorem{definition}[theorem]{Definition}
\newtheorem{example}[theorem]{Example}

\newtheorem{lemma}[theorem]{Lemma}

\newtheorem{proposition}[theorem]{Proposition}

\numberwithin{equation}{section}

\def\XXint#1#2#3{{\setbox0=\hbox{$#1{#2#3}{\int}$}
		\vcenter{\hbox{$#2#3$}}\kern-.5\wd0}}

\newcommand{\R}{\mathbb{R}}

\newcommand{\N}{\mathbb{N}}
\newcommand{\Z}{\mathbb{Z}}

\newcommand{\Ha}{\mathcal{H}}
\newcommand{\leb}{\mathcal{L}}

\newcommand{\spt}{\operatorname{spt}}

\newcommand{\Lip}{\operatorname{Lip}}
\newcommand{\LIP}{\operatorname{LIP}}

\newcommand{\ud}{\mathrm {d}}

\newcommand{\inv}{^{-1}}

\newcommand{\FillVol}{\operatorname{FillVol}}

\newcommand{\vol}{\operatorname{Vol}}
\newcommand{\Jac}{\operatorname{Jac}}
\newcommand{\md}{\operatorname{md}}

\newcommand{\norm}[1]{\lVert#1\rVert}
\newcommand{\abs}[1]{\lvert#1\rvert}
\newcommand{\on}{\:\mbox{\rule{0.1ex}{1.2ex}\rule{1.1ex}{0.1ex}}\:}
\newcommand{\bb}[1]{\llbracket #1\rrbracket}
\newcommand{\susbet}{\subset}

\DeclareMathOperator{\mass}{\mathbf{M}}
\DeclareMathOperator{\bI}{\mathbf{I}}
\DeclareMathOperator{\set}{set}
\DeclareMathOperator{\Ell}{\mathscr{L}}
\DeclareMathOperator{\ir}{ir}

\renewcommand{\rho}{\varrho}

\title[Filling minimality and rigidity of convex bodies]{Filling minimality and Lipschitz-volume rigidity of convex bodies among integral current spaces}


\author{Giuliano Basso}
\address{
Max Planck Institute for Mathematics,
Vivatsgasse 7,
53111 Bonn,
Germany}
\email{basso@mpim-bonn.mpg.de}

\author{Paul Creutz}
\address{
Max Planck Institute for Mathematics,
Vivatsgasse 7,
53111 Bonn,
Germany}
\email{creutz@mpim-bonn.mpg.de}

\author{Elefterios Soultanis}
\address{Radboud University, Heyendaalseweg 135, 6525 AJ, Nijmegen, Netherlands}
\email{elefterios.soultanis@gmail.com}

\date{\today}

\begin{document}
\begin{abstract}
In this paper we consider metric fillings of convex bodies. We show that convex bodies $C\subset \mathbb{R}^n$ are the unique minimal fillings of their boundary metrics among all integral current spaces. To this end, we also prove that convex bodies enjoy the Lipschitz-volume rigidity property within the category of integral current spaces, which is well known in the smooth category. As further applications of this result, we prove a variant of Lipschitz-volume rigidity for round spheres and answer a question of Perales concerning the intrinsic flat convergence of minimizing sequences for the Plateau problem.
\end{abstract}
\maketitle
\section{Introduction}
\subsection{Statement of main results}\label{sec:intro}
Let $Y$ be a closed orientable smooth manifold equipped with a compatible metric. Following Gromov \cite{gro87} the \emph{filling volume} $\FillVol_\infty(Y)$ is defined to be the infimum over the volumes of complete Riemannian manifolds $X$  which bound $Y$. A Riemannian manifold $M$ is called \emph{minimal filling} if $\vol^n(M)=\FillVol_\infty(\partial M)$, where $\partial M$ is equipped with the subspace metric. Calculating filling volumes and finding minimal fillings is notoriously difficult. Even the filling volume of simple spaces, such as $S^1$ endowed with the angular metric, is unknown. Whether the minimal filling in this case is the round hemisphere is called Gromov's filling area conjecture and, despite  remarkable partial results \cite{BCIK05, pu52, zust2021riemannian}, it remains widely open. 

The deep work of Burago--Ivanov \cite{bi10,bi13} shows that Riemannian manifolds which are $C^3$-close to full dimensional submanifolds of Euclidean or hyperbolic space are minimal fillings. See also \cite{ruan2022filling} for a recent generalization of their work to symmetric spaces of negative curvature. However not every $Y$ admits a smooth minimal filling as defined above. As a consequence of the Sormani–Wenger compactness theorem \cite{sw11} the situation changes when, in the definition of the filling volume, smooth Riemannian manifolds are replaced by integral current spaces $X$. The latter guarantees that $X$ carries an analytically defined object~$\bb{X}$ that one should think of as the fundamental class of $X$. In particular there are well defined notions of boundary and volume for such spaces, and hence it is possible to define the filling volume of an integral current space $Y$ as well as when an integral current space $X$ is a minimal filling, see e.g.\ \cite[Section 2.7]{ps17}.

The first main result in this paper is that convex bodies in $\mathbb{R}^n$ are 
the unique minimal fillings of their boundaries among all integral current spaces. Integral current spaces include, but are not limited to, compact oriented manifolds equipped with a metric that is bi-Lipschitz equivalent to a Riemannian (or Finsler) metric.

\begin{theorem}
\label{thm:main-i}

Let \(C\susbet \R^n\) be a convex body and suppose \(\iota\colon \partial C\to X\) is an isometric embedding into an integral current space \(X\) such that \(\iota_{\#}\bb{\partial C}=\bb{\partial X}\). Then
\begin{equation}\label{eq:main-ineq}
\mass^{\ir}({\bb{X}})\geq \vol^n(C)
\end{equation}
with equality if and only if the map \(\iota\) extends to an isometry $C\to X$.
\end{theorem}

Here $\mass^{\ir}(T)$ denotes the inscribed Riemannian mass of an integral current $T$, a variant of the usual mass $\mass(T)$ as introduced by Ambrosion--Kirchheim \cite{amb00}. We refer to Section~\ref{sec:rectifiable_currents} for the precise definitions and mention here that $\mass^{\ir}(\bb{X})$, resp.\ $\mass^{\ir}(\bb{\partial X})$, correspond to the the volume $\vol^n(X)$, resp. boundary volume $\vol^{n-1}(\partial X)$, when $X$ is a Riemannian manifold.  Furthermore one always has \(\mass(\bb{X}) \leq \mass^{\ir}(\bb{X})\). If we additionally assume that $X$ is infitesimally Euclidean, in the sense of having property (ET) \cite{lw17}, then the two notions of mass agree. We remark that the mass estimate \eqref{eq:main-ineq} remains true if \(\mass^{\ir}(\bb{X})\) is replaced by \(\mass(\bb{X})\); see Lemma~\ref{lem:absolute-filling-vol}. This answers a question by Sormani mentioned in her talk \cite[57:15-58:00]{sor22}. However, the substantial part of Theorem~\ref{thm:main-i} is the rigidity, which remains open for the Ambrosio--Kirchheim mass $\mass$.

The proof of Theorem~\ref{thm:main-i} is based on ideas of Burago--Ivanov \cite{bi10, bi13}. They use the following observation: \textit{Let $X$, $M$ be closed orientable Riemannian manifolds of the same dimension and $f\colon X\to M$ be a $1$-Lipschitz map of degree one. If $\vol^n(X)\leq \vol^n(M)$ then $f$ is a metric isometry.} Variants of this statement have been obtained and applied by Besson--Courtois--Gallot, Burago--Ivanov and Cecchini--Hanke--Schick \cite{bcg95, bi95, bi10, bi13, chs22}. Generalizations to the singular settings of Alexandrov and limit RCD spaces were obtained by Li and Li--Wang in \cite{li15, lw14}. See also \cite{li20} for an overview on these so-called \emph{Lipschitz-volume rigidity} results. The following variant, which is needed in the proof of Theorem~\ref{thm:main-i}, does not impose curvature assumptions on the domain space $X$ and is hence of independent interest.

\begin{theorem}\label{thm:rigidity-result-current-version}
Let \(X\) be an $n$-dimensional integral current space, $C\subset \mathbb{R}^n$ be a convex body and \(f\colon X\to \mathbb{R}^n\) be a \(1\)-Lipschitz map
such that \(f_\# \bb{\partial X}=\bb{\partial C}\). If \(\mass(\bb{\partial X})\leq \vol^{n-1}(\partial C)\) and \(\mass(\bb{X})\leq \vol^n(C)\), then \(f\) is an isometry.
\end{theorem}

Note that the condition on the degree of \(f\) is replaced by \(f_\# \bb{\partial X}=\bb{\partial C}\). Using Federer's constancy theorem, it is easy to see that this implies \(f_\# \bb{X}=\bb{C}\). 

As a corollary of Theorem~\ref{thm:rigidity-result-current-version} we derive the Lipschitz-volume rigidity of the round sphere $S^n$ with respect to the inscribed Riemannian area functional. Here \emph{round sphere} refers to the standard sphere $S^n$ endowed with its intrinsic metric as a Riemannian manifold. Indeed for simple reasons $S^n$ endowed with the subspace metric of $\mathbb{R}^{n+1}$ cannot be Lipschitz--volume rigid; see Section~\ref{sec:counterexamples-questions} below.

\begin{corollary}
\label{cor:sphere-rigidity}
Let \(X\) be an $n$-dimensional integral current space with $\partial X=\varnothing$ and \(f\colon X\to S^n\) be a \(1\)-Lipschitz map
such that \(f_\# \bb{X}=\bb{S^n}\). If \(\mass^{\ir}(\bb{X})\leq \vol^{n}(S^n)\), then \(f\) is an isometry.
\end{corollary}

Besides the proofs of Theorem~\ref{thm:main-i} and Corollary~\ref{cor:sphere-rigidity}, as another application of Theorem~\ref{thm:rigidity-result-current-version}  we answer the following question of Perales concerning the Euclidean unit ball $B^n$. This question was promoted in the same talk by Sormani \cite[53:12--55:40]{sor22}. 

\begin{question}[Perales]
\label{question:Perales-Sormani}
Assume $(M_i)$ is a sequence of compact orientable Riemannian $n$-manifolds with 
$$
\lim_{i\to \infty} \vol^n (M_i)\leq \vol^n(B^n)
$$ 
which converges in the intrinsic flat sense to a limit space $X$. Assume further that $f_i\colon M_i\to \mathbb{R}^N$ are $1$-Lipschitz maps such that $f_{i\#}\bb{\partial M_i}=\bb{S^{n-1}}$ for all \(i\in\N\), the $f_i$ take values in a compact set $K\subset \mathbb{R}^N$, and $(f_i)$ converges in the sense of the Sormani Arzelà--Ascoli theorem~\cite{sor18} to a limit map $f\colon X\to K$. Does it follow that \(f\) is an isometry $X\to B^n$?
\end{question}

Informally, Question \ref{question:Perales-Sormani} asks about the interplay of the `extrinsic' flat convergence of integral currents in $\mathbb{R}^N$ and the intrinsic flat convergence of the corresponding `intrinsic metrics' on the currents.  In general, Question~\ref{question:Perales-Sormani} has a negative answer, see Example~\ref{ex:counterexample} below. However, as a consequence of Theorem ~\ref{thm:rigidity-result-current-version} the answer is positive if one 
imposes a suitable bound on the boundary volumes.
\begin{corollary}\label{prop:answer-to-sormani-perales}
Let $C\subset \R^n$ be a convex body. Suppose \((X_i)\) is a sequence of integral current spaces with
\[
\liminf_{i\to \infty} \mass(\bb{\partial X_i}) \leq \vol^{n-1}(\partial C) \quad \textnormal{and}\quad \liminf_{i\to \infty} \mass(\bb{X_i}) \leq \vol^n(C)
\]
which converges in the intrinsic flat sense to a limit space $X$. Assume further that \(f_i \colon X_i \to \R^N\) are \(1\)-Lipschitz maps such that \(f_{i\#}\bb{\partial X_i}\) flat converges to \(\bb{\partial C}\), the \(f_i\) take values in a compact set \(K\susbet \R^N\), and \((f_i)\) converges in the sense of Theorem~\ref{thm:Sormani-arzela} to a \(1\)-Lipschitz map \(f\colon X\to K\). Then \(f\) is an isometry \(X \to C\). 
\end{corollary}

\subsection{Strategy of proof} We prove the filling volume rigidity Theorem \ref{thm:main-i} by deducing it from the Lipschitz-volume rigidity Theorem \ref{thm:rigidity-result-current-version} following the arguments of Burago--Ivanov \cite{bi10,bi13}. The key idea is to consider a linear isometric embedding $\Phi\colon\R^n\to \Ell$ into the injective Banach space $\Ell:=L^{\infty}(S^{n-1})$. There is an inner product on $\Ell$ whose induced norm agrees with $\|\cdot\|_\infty$ on $\Phi(\R^n)$ and which does not increase the inscribed Riemannian masses of rectifiable currents in $\Ell$ (see Lemma \ref{lem:proof-of-main-thm}). The inner product induces an orthogonal projection onto $\Phi(\mathbb{R}^n)$ which, together with the injectivity of $\Ell$,  implies the existence of a map $f\colon X \to \mathbb{R}^n$ with $f_{\#} \bb{\partial X}=\bb{\partial C}$ that does not increase inscribed Riemannian volumes.  Although the argument of Burago--Ivanov showing that $f$ is 1-Lipschitz does not generalize directly to integral current spaces, Theorem~\ref{thm:main-i} follows from a double application of Theorem~\ref{thm:rigidity-result-current-version} which circumvents this problem.

To prove our Lipschitz-volume rigidity result, Theorem \ref{thm:rigidity-result-current-version}, we use a recently established decomposition result for $1$-dimensional currents \cite{Bonicatto-Del-Nin-Pasqualetto-2022} to obtain the case $n=1$. The general case \(n\geq 2\) can be  reduced to the 1-dimensional case by considering suitable slicings of the current \(T=\bb{X}\). Let \(H\susbet \R^n\) be an \((n-1)\)-dimensional hyperplane and \(\rho \colon \R^n \to H\) the orthogonal projection onto \(H\). We define \(\widehat{\rho}=\rho \circ f\). The slices \(T_p=\langle T, \widehat{\rho}, p \rangle\),
which are defined for \(\Ha^{n-1}\)-almost every \(p\in H\), are \(1\)-dimensional integral currents and satisfy \(f_\# T_p =\bb{C_p}\), where \(C_p= C \cap \rho^{-1}(p)\) is isometric to a closed interval in \(\R\). We prove that the integral current space \(X_p=(\set T_p, T_p)\) satisfies 
\[
f_\# \bb{X_p}=\bb{C_p}, \quad \mass(\bb{X_p})\leq \Ha^{1}(C_p), \quad \text{and} \quad \mass(\bb{\partial X_p})\leq \Ha^{0}(\partial C_p).
\]
The crucial step of the proof is to show the last inequality concerning the mass of \(\bb{\partial X_p}\). This follows essentially from Lemmas~\ref{lem:mass-of-preimages} and \ref{lem:mass-of-images}, which show that \(f\) is mass-preserving in a certain sense. Thus we may deduce from the case $n=1$ that \(f\) is an isometry $\spt \langle T ,\widehat{\rho}, p \rangle \to \spt \langle \bb{C} ,\rho,p \rangle$ for \(\Ha^{n-1}\)-almost every \(p\in H\).  By Proposition~\ref{lemm:step-1} below this suffices to conclude that \(f\) is an isometry.

To deduce Corollary~\ref{cor:sphere-rigidity} we verify that the assumptions of Theorem~\ref{thm:rigidity-result-current-version} are satisfied for the coning map $Cf\colon CX\to C S^n=B^{n+1}$.

\subsection{Organization}The paper is organized as follows. After reviewing definitions and well established facts concerning measure theory, volume functionals and metric currents in Section~\ref{sec:2}, we prove the basic properties of mass preserving Lipschitz maps needed in this paper in Section~\ref{sec:3}. 

The proof of Theorem \ref{thm:rigidity-result-current-version} is given in Section~\ref{sec:4}, starting with the special 1-dimensional case in Section~\ref{sec:case-n=1}, the reduction to the 1-dimensional case in Section~\ref{sec:4.2}, and the conclusion in Section~\ref{sec:proof-of-main-thm}. Section~\ref{sec:5} is devoted to the proof of Corollary \ref{cor:sphere-rigidity}, while Theorem \ref{thm:main-i} is proved in Section~\ref{sec:6}, with the filling estimate proved in Section~\ref{sec:6.1}, and the rigidity statement in Section~\ref{sec:6.2}. In Section~\ref{sec:7.1} we discuss intrinsic flat convergence, while Section~\ref{sec:7.2} is devoted to the counterexample to the question of Perales and the proof of Corollary~\ref{prop:answer-to-sormani-perales}. Lastly, in Section~\ref{sec:8}, we discuss possible extensions of our results and further counterexamples.

\subsection{Acknowledgements} We would like to thank Alexander Lytchak for bringing Question~\ref{question:Perales-Sormani} to our attention. We are also grateful to Alexander Lytchak, Giacomo Del Nin, Raquel Perales, and Roger Züst for several helpful remarks.

\bigskip

\emph{
After the completion of the first version of this paper, we were informed by Raquel Perales that she and Giacomo Del Nin have obtained a similar result to Theorem \ref{thm:rigidity-result-current-version} independently. In their work \cite{Upcoming} (which has since appeared on the arXiv) they discuss in detail the motivation for Question~\ref{question:Perales-Sormani}, which is to give a direct argument for a gap in the proof of \cite[Theorem 1.3]{sormani-2017}.  See also \cite{sormani-corrigendum-2022} and \cite{allen2020intrinsic, allen2020volume,HLP20}.}


\section{Preliminaries}\label{sec:2}
\subsection{Basic notation and definitions} We write \(\N=\{ 1, 2, \dots\}\) for the set of positive integers. Moreover, we let \(\R^n\) denote the set of \(n\)-tuples of real numbers with the convention that \(\R^0\) consists of exactly one point.  
Let \(X=(X, d)\) be a metric space. We denote by $B_X(x,r)$ (or simply $B(x,r)$) the closed ball of radius $r$ centered at $x$, and by $B^n$ the $n$-dimensional Euclidean unit ball $B_{\mathbb{R}^n}(0,1)$.  Unless otherwise specified, subsets of \(X\) are always endowed the subspace 
metric. We write \(\overline{X}\) for the metric completion of \(X\) and we tacitly identify \(X\) with its canonical isometric copy in \(\overline{X}\). A map \(f\colon X\to Y\) between metric spaces \(X\), \(Y\) is called \emph{\(L\)-Lipschitz}, for some constant \(L\geq 0\), if 
\(
d(f(x), f(y)) \leq L d(x,y)
\)
for all \(x\), \(y\in X\). The smallest \(L\geq 0\) such that \(f\) is \(L\)-Lipschitz is denoted by \(\Lip f\). We use $\LIP(X)$ to denote the set of Lipschitz functions $X\to \R$. A metric space $Y$ is called \emph{injective} if whenever $X$ is a metric space, $A\subset X$ and $f\colon A\to Y$ is a $1$-Lipschitz map, then $f$ can be extended to a $1$-Lipschitz map $\bar{f}\colon X\to Y$.  A well-known theorem of McShane (see e.g. \cite[Theorem 1.27]{brudnyi2012}) states that \(\R\) is an injective metric space. Hence, each of the Banach spaces \(\ell_\infty^n\coloneqq(\R^n, \norm{\cdot}_\infty)\) is injective as well. We say that \(X\) and \(Y\) are \emph{bi-Lipschitz equivalent} if there exists a bijection \(f\colon X \to Y\) such that \(f\) and \(f^{-1}\) are both Lipschitz maps.  A separable metric space \(X\) is called \emph{Lipschitz \(n\)-manifold} if every \(x\in X\) has a closed neighborhood which is bi-Lipschitz equivalent to \(B^n\).

\subsection{Measure theory}

As is common in geometric measure theory, we follow the convention that a \emph{measure}  on  a metric space $X$ is a countably subadditive function $\mu \colon \mathcal{P}(X)\to [0,\infty]$ such that $\mu (\emptyset)=0$. Throughout the forthcoming discussion $\mu$ will always be a measure on $X$. We say that $A\subset X$ is \emph{$\mu$-measurable} if 
\[
\mu(M\cap A)+\mu(M\cap A^c)=\mu(M)
\]
for all $M\subset X$. The collection of all $\mu$-measurable subsets of $X$ forms a $\sigma$-algebra and the restriction of $\mu$ to this $\sigma$-algebra is countably additive. The measure $\mu$ is \emph{Borel} if all Borel subsets of $X$ are $\mu$-measurable. Furthermore, $\mu$ is \emph{Borel regular} if $\mu$ is Borel and for every $M\subset X$ there is a Borel set $A\subset X$ such that $M\subset A$ and $\mu(M)=\mu(A)$. If $X$ is separable and a Borel subset of its completion \(\overline{X}\), then every finite Borel measure $\mu$ on $X$ is \emph{tight}, see \cite[Theorem 3.2]{par67}). The latter means that for every Borel set $A\subset X$ and $\varepsilon >0$, there is a compact set $K\subset A$ with $\mu(A\setminus K)<\varepsilon$. The \emph{support} of $\mu$ is the closed set $\spt \mu$ which consists of those $x\in X$ such that \(\mu(B(x,r))>0\) for every \(r>0\). If $\mu$ is a finite Borel measure then its support is separable. We say that $\mu$ is \emph{concentrated} on $A\subset X$ if $\mu(X\setminus A)=0$. If $X$ is separable, then $\mu$ is concentrated on its support (see \cite[Theorem 2.2.16]{fed69}). Indeed it is consistent with (but not implied by) ZFC that the latter holds true without assuming $X$ to be separable (compare Section~2.1.6 and Theorem~2.2.16 in \cite{fed69}). In particular, if one assumes this additional set-theoretic axiom, finite Borel measures on \emph{every} complete metric space are tight. In \cite{amb00, Bonicatto-Del-Nin-Pasqualetto-2022} this is a standing assumption because there the authors want to treat also currents in non-separable metric spaces $X$, and even for separable $X$ some arguments therein rely on embedding $X$ isometrically into the non-separable Banach space $\ell^\infty$. For the proof of Theorem \ref{thm:rigidity-result-current-version} it can be avoided to assume this axiom. But since we will not justify the use of auxiliary results from the articles \cite{amb00, Bonicatto-Del-Nin-Pasqualetto-2022}, the reader is invited to also consider it an additional standing assumption throughout this paper.

\subsection{Volumes of rectifiable spaces}
\label{subsec:jacobians}
By Caratheodory's criterion (see e.g.\ \cite[p.\ 75]{fed69}) the Hausdorff $n$-measure on~$X$, which will be denoted by $\mathcal{H}^n_X$ (or simply~$\mathcal{H}^n$), is Borel regular. In this paper, following a common convention, we normalize $\Ha^n_X$ so that $\Ha^n_{\R^n}$ equals the Lebsegue measure~$\leb^n$. Hausdorff measures can sometimes be calculated in terms of the so-called area formula. The area formula relies on the following metric version of the Rademacher theorem due to Kirchheim~\cite{kir94}: \emph{if $f\colon A \to X$ is a Lipschitz map from a Borel set $A\subset \mathbb{R}^n$ then, for almost every $p\in A$, there exists a seminorm $\md f_p$ on $\mathbb{R}^n$ such that
\[
d(f(q),f(p))=\md f_p(q-p)+ o\left(|q-p|\right)
\]
as $q\to p$ with $q\in A$.} Now the area formula \cite{kir94} states that the function $x \mapsto \#\{f^{-1}(x)\}$ is $\mathcal{H}^n_X$-measurable and
\begin{equation}
\label{eq:area-formula}
\int_A \Jac^b(\md f_p) \ \textrm{d}p=\int_X \#\left\{f^{-1}(x)\right\}\ \textrm{d}\mathcal{H}^n(x).
\end{equation}
Here $\Jac^b(\sigma)$ denotes the \emph{Busemann Jacobian} of a seminorm $\sigma \colon \mathbb{R}^n\to [0,\infty)$ which is defined as  $\omega_n/ \mathcal{L}^n(B_\sigma)$ where $B_\sigma$ is the unit ball of $\sigma$ and \(\omega_n\coloneqq \leb^{n}(B^n)\). More generally, a map $\Jac^{\bullet}:\Sigma^n\to [0,\infty)$ from the space $\Sigma^n$ of seminorms on $\R^n$ is a \emph{Jacobian} if 
\begin{itemize}
	\item[(i)] $\Jac^\bullet(\abs{\cdot})=1$ for the standard Euclidean norm \(\abs{\cdot}\) on \(\R^n\); 
	\item[(ii)] $\Jac^\bullet(\sigma^\prime)\le \Jac^\bullet(\sigma)$ if $\sigma^\prime\le \sigma$;
	\item[(iii)] $\Jac^\bullet(\sigma\circ T)=|\det T|\Jac^\bullet(\sigma)$ for any linear map $T:\R^n\to\R^n$.
\end{itemize}
The properties above are known as \emph{normalization, monotonicity,} and \emph{transformation law}, respectively. See \cite[Section 4.1]{cre20} and the references therein for a more detailed overview. Two Jacobians which play an important role in this paper are  Gromov's \emph{mass\(*\) Jacobian}, defined by \[
\Jac^{m*}(\sigma)\coloneqq \sup_P \frac{2^n}{\mathcal{L}^n(P)}
\] where the supremum is taken over all parallelepipeds \(P\) containing $B_\sigma$, and Ivanov's \emph{inscribed Riemannian Jacobian}
\[
\Jac^{\ir}(\sigma)\coloneqq\frac{\omega_n}{\leb^n(J(B_\sigma))},
\]
where $J(B_\sigma)\subset \mathbb{R}^n$ is the \emph{John ellipsoid} of $B_\sigma$, that is, the ellipsoid of maximal $\leb^{n}$-measure contained in $B_\sigma$. It follows from John's theorem (compare e.g.\ \cite{Ball92}) that if $\sigma_1\in \Sigma^{n_1}$ and $\sigma_2\in \Sigma^{n_2}$ then 
\begin{equation}
\label{eq:ir-product}
    \Jac^{ir}(\sigma_1\times\sigma_2)=\Jac^{ir}(\sigma_1)\cdot \Jac^{ir}(\sigma_2)
\end{equation}
where $\sigma_1\times \sigma_2 \in \Sigma^{n_1+n_1}$ is defined by $(\sigma_1\times \sigma_2)(v_1,v_2)=\sqrt{\sigma_1(v_1)^2+\sigma_2(v_2)^2}$.

$X$ is called \emph{$n$-rectifiable} if there are Borel subsets $\{A_i\}_{i\in \mathbb{N}}$ of $\mathbb{R}^n$ and bi-Lipschitz embeddings $\{\varphi^i \colon A_i \to X\}_{i\in \mathbb{N}}$ such that $\mathcal{H}^n_X$ is concentrated on $\bigcup_{i\in \mathbb{N}}\varphi^i(A_i)$. Without loss of generality, one can assume additionally that $\varphi^i(A_i)$ and $\varphi^j(A_j)$ are disjoint if $i\neq j$. For $n$-rectifiable spaces, the \emph{density}
\begin{align*}
	\Theta_n(B,x)= \lim_{r\to 0}\frac{\Ha^n(B\cap B(x,r))}{ \omega_n r^n}
\end{align*}
exists and is equal to one for $\Ha^n$-almost every $x\in B$. Morever, by the area formula \eqref{eq:area-formula} the Hausdorff $n$-measure of a Borel set $B\subset X$ is given by
\begin{equation*}
\mathcal{H}^n(B)=\sum_{i\in \mathbb{N}}\ \int_{(\varphi^i)^{-1}(B)} \ \Jac^b(\md \varphi^i_p) \ \textrm{d}p.
\end{equation*}

When $X$ has `non-Euclidean tangents' (i.e.\ the metric differentials $\md \varphi_p^i $ are not necessarily induced by inner products) different Jacobians yield distinct notions of volume. Indeed, every Jacobian $\Jac^{\bullet}$ gives rise to a Borel regular measure $\mu^{\bullet}_X$ on $X$ by setting
\begin{equation}\label{eq:jac_volume}
\mu^{\bullet}_X(B)\coloneqq\sum_{i\in \mathbb{N}}\ \int_{(\varphi^i)^{-1}(B)} \ \Jac^{\bullet}(\md \varphi_p^i) \ \textrm{d}p
\end{equation}
for Borel sets $B\subset X$. It follows from the chain rule for metric differentials and the transformation law (iii) that this does not depend on the choice of the coordinate charts $\{\varphi^i\}$. Furthermore by the normalization axiom (i) of Jacobians and \eqref{eq:area-formula} one always has $\mu^{\bullet}_{\mathbb{R}^n}=\mathcal{H}^n_{\mathbb{R}^n}$. In Section \ref{sec:rectifiable_currents} we will see an analogous construction for the mass measure of rectifiable currents.

\subsection{Metric currents}\label{sec:metric-currents}
Using ideas of De Giorgi \cite{de-giorgi-1995} and extending the classical theory of currents, which goes back to de Rham and Federer--Fleming \cite{de-Rham-1955,fed60}, Ambrosio--Kirchheim  introduced \emph{metric currents} in \cite{amb00}. Variants of their definitions have been proposed and studied by several authors (see \cite{lan11,lang-wenger-2011,williams-2012,zust-2019}). In this paper we follow the original approach of Ambrosio--Kirchheim and review its basic aspects below.

For each \(n\geq 0\) we let \(\mathcal{D}^n(X)\) denote the set of all tuples \((h, \pi_1, \dots, \pi_n)\), where \(h\colon X\to \R\) is a bounded Lipschitz function and \(\pi_i\in \LIP(X)\). 
\begin{definition}\label{def:current}
Let $X$ be a complete metric space. An \((n+1)\)-multilinear map \(T\colon \mathcal{D}^n(X)\to \R\) is called \(n\)-current if the following holds.

\begin{enumerate}
    \item\label{ax:1} \(T(h, \pi_1^{(j)}, \dots, \pi_n^{(j)}) \to T(h, \pi_1, \dots, \pi_n)\) as \(j\to \infty\), whenever \(\pi_i^{(j)}\to \pi_i\) pointwise and \(\Lip \pi_i^{(j)} \leq C\) for some uniform constant \(C\).
    \item\label{ax:2} \(T(h, \pi_1, \dots, \pi_n)=0\) if there is \(i\in \{1, \dots, n\}\) such that \(\pi_i\) is constant when restricted to an open neighborhood of \(\{ x\in X : h(x)\neq 0\}\).
    \item\label{ax:3} There is a finite Borel measure \(\mu\) on \(X\) such that
    \begin{equation}\label{eq:mass-ineq}
            \abs{T(h, \pi_1, \dots, \pi_n)} \leq \prod_{i=1}^n \Lip \pi_i \, \int_X \abs{h} \, d\mu
    \end{equation}
    for all \((h, \pi_1,\dots, \pi_n)\in \mathcal{D}^n(X)\).
\end{enumerate}
\end{definition}

The minimal measure $\mu$ satisfying \eqref{eq:mass-ineq} is called the \emph{mass} of $T$ and is denoted by $\norm{T}$. Any \(n\)-current $T$ extends to a functional $T:L^1(X, \|T\|)\times \LIP(X)^n\to\R$ satisfying \eqref{eq:mass-ineq}. We define \(\mass(T)\coloneqq\norm{T}(X)\), \(\spt T\coloneqq \spt \norm{T}\), and write \(\mass_n(X)\) for the vector space of all \(n\)-currents on \(X\). It is easy to check that \(\mass_n(X)\) becomes a Banach space when it is endowed with the norm \(\mass(\cdot)\). 
There are natural push-forward, restriction and boundary operators on \(\mass_n(X)\), which we recall next.

Every Lipschitz map \(f\colon X\to Y\) between complete metric spaces \(X\), \(Y\) induces a \emph{push-forward} map \(f_\# \colon \mass_n(X) \to \mass_n(Y)\) on the level of currents. Indeed, for every \(T\in \mass_n(X)\) we define
\[
f_\# T(h, \pi_1, \dots, \pi_n)=T(h\circ f, \pi_1\circ f, \dots, \pi_n \circ f)
\]
for all $(h, \pi_1, \dots, \pi_n)\in \mathcal{D}^n(Y)$. In particular we note that 
\begin{equation}
\label{eq:mass-pf-inequality}
\norm{f_{\#}T}\leq (\Lip f)^n f_{\#}\norm{T}.
\end{equation}
If $f\colon X\to Y$ is a Lipschitz map between arbitrary metric spaces, then $f$ extends to a unique Lipschitz map $\bar{f}\colon \overline{X}\to\overline{Y}$. By abuse of notation, we will usually write $f_\#$ instead of $\bar{f}_\#$.

Given an $n$-current $T\in \mass_n(X)$, \(\ell\in \{0, \dots, n\}\), and an $(\ell+1)$-tuple \(\omega=(g, \omega_1, \dots, \omega_{\ell})\), where \(g\colon X\to \R\) is a bounded Borel function and \(\omega_i\in \LIP(X)\), the \emph{restriction} $T\on\omega$ of $T$ by $\omega$ is an \((n-\ell)\)-current defined by
\[
T \on \omega \,(h, \pi_1, \dots, \pi_{n-\ell})=T(h \, g, \omega_1, \dots, \omega_{\ell}, \pi_1, \dots, \pi_{n-\ell})
\]
for all \((h, \pi_1, \dots, \pi_{n-\ell})\in \mathcal{D}^{n-\ell}(X)\). In particular, \(T\on A \coloneqq T \on \mathbbm{1}_A\), is a well-defined \(n\)-current for any Borel set \(A\subset X\).

Finally, if \(n\geq 1\) the \emph{boundary} \(\partial T\) of \(T\in \mass_n(X)\) is the \(n\)-multilinear map \(\partial T \colon \mathcal{D}^{n-1}(X)\to \R\) defined by
\[
\partial T(h, \pi_1, \dots, \pi_{n-1})=T(1, h, \pi_1, \dots, \pi_{n-1}).
\]
We say that \(T\) is a normal current if $\partial T\in \mass_{n-1}(X)$. The vector space of all normal \(n\)-currents on \(X\) is denoted by \(\mathbf{N}_n(X)\) and we set \(\mathbf{N}_0(X)=\mass_0(X)\). The spaces $\mathbf{N}_n(X)$ equipped with the norm \(\mathbf{N}(T)=\mass(T)+\mass(\partial T)\) are Banach spaces, with the convention \(\mathbf{N}(T)=\mass(T)\) if \(T\in \mathbf{N}_0(X)\).

\subsection{Rectifiable currents and their Finsler mass}\label{sec:rectifiable_currents}
Every \(\theta \in L^1(\R^n)\) induces an \(n\)-current \(\bb{\theta}\in \mass_n(\R^n)\) given by
\[
\bb{\theta}(h,\pi_1 \dots, \pi_n)=\int_{\R^n} \theta h  \det [\partial_i \pi_j]^n_{i, j=1} \, d \mathcal{H}^n
\]
for all \((h, \pi_1, \dots, \pi_n)\in \mathcal{D}^n(\R^n)\). We say that $T\in \mass_n(X)$ is \emph{rectifiable} (resp.\ \emph{integer-rectifiable}) if there are compact sets $K_i \subset \mathbb{R}^n$, functions $\Theta_i \in L^1(\R^n)$ (resp.\ $\Theta_i \in L^1(\R^n; \mathbb{Z})$) 
with $\spt \Theta_i \subset K_i$ and bi-Lipschitz embeddings $\varphi_i \colon K_i \to X$ such that
\begin{equation}\label{eq:rep-rect-current}
T=\sum_{i\in \N} \varphi_{i\#} \bb{\Theta_i} \ \ \textnormal{and} \ \ \mass(T)=\sum_{i\in \N} \mass \left(\varphi_{i\#} \bb{\Theta_i}\right).
\end{equation}
We denote by $\mathcal{R}_n(X)$ (resp. $\mathcal{I}_n(X)$) the collection of all rectifiable (resp. integer-rectifiable) currents on $X$. The mass of a rectifiable current $T$ has the following very concrete interpretation in terms of the Gromov mass* volume $\mu^{m*}$,
\begin{equation}
\label{eq:mass*-characterization}
\norm{T}(A)=\sum_{i\in \N} \int_{A\cap \varphi_i(K_i)} \left|\Theta_i \circ \varphi^{-1}_i (x)\right|\ \textrm{d}\mu^{m*}(x)
\end{equation}
for every Borel set \(A\susbet X\) (see, for example, \cite[Lemma~2.5(2)]{zust2021riemannian}).  More generally, given a Jacobian $\Jac^\bullet$ and the associated volume measure $\mu_X^\bullet$, the \emph{Finsler mass measure} $\|T\|^\bullet$ is defined by 
\begin{equation}
\label{eq:Jacobian-mass-characterization}
\norm{T}^\bullet(A)\coloneqq\sum_{i\in \N} \int_{A\cap \varphi_i(K_i)} \abs{\Theta_i\circ \varphi_i\inv(x)} \ \textrm{d}\mu_X^\bullet
\end{equation}
for every Borel subset \(A\susbet X\). It can be shown that this definition is independent of the chosen representation \eqref{eq:rep-rect-current}. It moreover satisfies natural estimates e.g. \(\norm{f_\#T}^\bullet\le (\Lip f)^n\norm{T}^{\bullet}\) for every Lipschitz map \(f\colon X \to Y\), and is comparable to the usual mass measure
\begin{align*}
\norm{T}^{\bullet} \leq \norm{T}^{\ir},\quad\textrm{and}\quad C\inv\cdot\norm{T}\le \norm{T}^{\bullet}\le C\cdot\norm{T}
\end{align*}
for a constant $C>0$ depending only on $n$, see \cite[Lemma~2.5]{zust2021riemannian}. In particular, one has \(\spt \norm{T}^\bullet=\spt T\). We call \(\mass^\bullet(T)\coloneqq \norm{T}^\bullet(X)\) the \emph{Finsler mass} associated to the Jacobian $\Jac^\bullet$.

We remark that $T\in \mathcal R_n(X)$ if and only if $\|T\|$ is concentrated on an $n$-rectifiable set and $\|T\|\ll \Ha^n$. Moreover, 
if \(T\in \mathcal{R}_n(X)\), then \(\norm{T}\) is concentrated on the \(n\)-rectifiable set
\begin{equation}\label{eq:characteristic-set}
\set T = \Bigl\{ x\in X : \liminf_{r\downarrow 0} \frac{\norm{T}(B(X, r))}{\omega_n r^n} >0\Bigr\},
\end{equation}
and any Borel set \(A\susbet X\) on which \(\norm{T}\) is concentrated contains \(\set T\) up to a \(\mathcal{H}^n\)-negligible set (see \cite[Theorem 4.6]{amb00} ). The set defined in \eqref{eq:characteristic-set} is called \emph{characteristic set} of \(T\).

\subsection{Slicing}\label{sec:slicing}
Let \(T\in \mathcal R_n(X)\)  and \(\rho \colon X\to \R^k\) be a Lipschitz map with \(k\in \{1, \dots, n\}\). In \cite[Theorems 5.6 and 5.7]{amb00}, Ambrosio and Kirchheim show that there is a natural slicing operator \(\R^k\ni p \mapsto \langle T, \pi, p\rangle \in \mathcal R_{n-k}(X)\) which is defined for \(\mathcal{H}^k\)-almost every \(p\in \R^k\). Each slice \(\langle T, \rho, p\rangle\) is concentrated on \(\spt T\cap  \rho^{-1}(p)\), and for every \(\psi \in  C_c(\R^k)\),
\begin{equation}\label{eq:universal-property}
\int_{\R^k} \langle T, \rho, p\rangle \, \psi(p) \, dp= T \on (\psi \circ \rho, \rho_1, \dots, \rho_k),
\end{equation}
where \(\rho_i\) denotes the \(i\)th coordinate function of \(\rho\).  Moreover, 
\begin{equation}\label{eq:mass-of-slices}
\int_{\R^k} \norm{\langle T, \rho, p\rangle} \, dp= \norm{T \on (1, \rho)},
\end{equation}
where \((1, \rho)\) is shorthand for \((1, \rho_1, \dots, \rho_k)\). In particular, the following slicing inequality holds true,
\[
\int_{\R^k} \mass{(\langle T, \rho, p\rangle)} \, dp \leq (\Lip \rho)^k \mass(T). 
\]
These properties uniquely characterize the slices \(\langle T, \rho, p\rangle\). Indeed, if \(T^p\in \mass_{n-k}(X)\) are concentrated on \(L \cap \rho^{-1}(p)\) for some \(\sigma\)-compact set \(L\), satisfy \(\int_{\R^k} \mass(T^p) \, dp < \infty\) and \eqref{eq:universal-property}, then \(T^p=\langle T, \rho, p\rangle\) for \(\Ha^k\)-almost every \(p\in \R^k\). Hence, for example, one has that the slicing and the push-forward operator commute. More concretely, if \(f\colon X\to Y\) and \(\rho\colon Y \to \R^k\) are Lipschitz maps, then
\begin{equation}\label{eq:slice-and-push-forward-commute}
f_\# \langle T, \widehat{\rho}, p \rangle=\langle f_\# T, \rho, p\rangle
\end{equation}
for \(\Ha^k\)-almost every \(p\in \R^k\), where \(\widehat{\rho}=\rho \circ f\).  

Naively one might suspect that \( \spt \langle T, \rho, p\rangle = \spt T\cap  \rho^{-1}(p)\) up to a set of \(\Ha^{n-k}\)-measure zero. However, as the following well-known example shows this cannot be true in general.
\begin{example}\label{ex:spt-vs-set}
Fix \(n\geq 1\) and let \(\{ x_i : i\in \N\}\) be a dense subset of \(\R^{n+1}\). Further,  let \((r_i)_{i\in \N}\) be a sequence of positive real numbers such that \(\sum_{i\in \N} r_i^n\) is finite. We put \(T_i=\partial \bb{B(x_i, r_i)}\) and \(T=\sum_{i\in\N} T_i\). By construction, \(T\) is an integer-rectifiable \(n\)-current. Moreover, it is easy to check that \(\spt T=\R^{n+1}\), and thus 
\[
\spt T \cap  \rho^{-1}(p)\cong \R
\]
for every orthogonal projection \(\rho\) onto a hyperplane. But \(\langle T, \rho, p\rangle\) is an integer-rectifiable \(0\)-current for \(\mathcal{H}^n\)-almost every \(p\in \R^n\). In particular, the support of \(\langle T, \rho, p\rangle\) consists of finitely many points, and so it cannot be equal to  \(\spt T\cap  \rho^{-1}(p)\) up to a set of \(\mathcal{H}^0\)-measure zero.
\end{example}

The following lemma shows that such an equality is true if instead of \(\spt \langle T, \rho , p \rangle\) and \(\spt T\) the corresponding characteristic sets are considered. 

\begin{lemma}\label{lem:support-of-sliced-measure-equals-slice}
If \(T\in \mathcal R_n(X)\) and \(\rho \colon X\to \R^k\) with \(k\leq n\) is a Lipschitz map, then up to a set of \(\Ha^{n-k}\)-measure zero
\begin{equation}\label{eq:support-of-slices}
\set\, \langle T, \rho, p \rangle = \set T \cap \rho^{-1}(p)
\end{equation}
for \(\Ha^k\)-almost every \(p\in \R^k\).
\end{lemma}

\begin{proof}
Without loss of generality we may assume that \(X=\spt T\).  
Furthermore, using Kuratowski's embedding, that \(X\subset \ell^\infty\) and hence \(T\in \mathcal{R}_n(Y)\), where \(Y= \ell^\infty\). In what follows, we combine different results from \cite{amb00} to obtain the desired equality. 
By \cite[Theorem~9.1]{amb00} there exist a \(\Ha^n\)-rectifiable set \(S\subset Y\), a Borel function \(\theta\colon S \to \R\), and an orientation \(\tau\) of \(S\) such that \(T=\bb{S, \theta, \tau}\). 

Now, \cite[Theorem~9.5]{amb00} implies that \(S\) and \(\set T\) are equal up to \(\Ha^n\)-negligible sets, that is,  \(S\cup N_1= \set T \cup N_2\), where \(\Ha^n(N_i)=0\). 
Because of the coarea inequality \cite{EH21}, \(\Ha^{n-k}( N_i \cap \rho^{-1}(p))=0\) for \(\Ha^k\)-almost every \(p\in \R^k\), and so for any such \(p\) we have \(S\cap \rho^{-1}(p)=\set T \cap \rho^{-1}(p)\) up to a set of \(\Ha^{n-k}\)-measure zero. 
By virtue of \cite[Theorem~9.7]{amb00}, for \(\Ha^k\)-almost every \(p\in \R^k\) there exists an orientation \(\tau_p\) of \(S\cap \rho^{-1}(p)\) such that \(\langle T, \rho, p \rangle = \bb{S \cap \rho^{-1}(p), \theta, \tau_p}\). Hence, up to a set of \(\Ha^{n-k}\)-measure zero, \(\set\, \langle T, \rho, p \rangle = S\cap \rho^{-1}(p)=\set T \cap \rho^{-1}(p)\), as desired. 
\end{proof}

Notice that if $k=n$ then \eqref{eq:support-of-slices} is an actual equality, since the empty set is the only \(\mathcal{H}^0\)-null set.

\subsection{Integral current spaces}
The space of integral currents \(\bI_n(X)\) is defined as
\[
\bI_n(X)=\mathcal{I}_n(X) \cap \mathbf{N}_n(X).
\]
Integral currents are the most important class of currents in this article. The seminal boundary-rectifiability theorem \cite[Theorem 8.6]{amb00} states that $\partial T\in \bI_{n-1}(X)$ whenever $T\in \bI_n(X)$ and \(n\geq 1\). We also remark that \(\bI_n(X)\) is a closed additive subgroup of \(\mathbf{N}_n(X)\) for every \(n\geq 0\) and, if \(T\in \bI_n(X)\), then \(f_\# T\in \bI_n(Y)\) for every Lipschitz map \(f\colon X\to Y\). Moreover, for any Lipschitz map $\rho:X\to \R^k$ one has $\langle T,\rho,p\rangle\in \bI_{n-k}(X)$ for \(\Ha^k\)-almost every $p\in \R^k$. The following definition is due to Sormani and Wenger (see \cite[Definition 2.46]{sw11}).

\begin{definition}[Integral current space]\label{def:int_current_space}
A pair \((X, T)\) is called \(n\)-dimensional integral current space if \(X\) is a metric space and \(T\in \bI_n(\overline{X})\) is such that \(\set T =X\). The current \(T\) is often denoted by \(\bb{X}\) and we generally do not emphasize the dependence of the integral current space \((X, T)\) on \(T\) and denote it only by \(X\).
\end{definition}

To any integral current space \((X, T)\), one can naturally associate a boundary \(\partial X=(\set \partial T, \partial T)\), which is also an integral current space. Prime examples of integral current spaces are compact connected orientable Lipschitz manifolds.
\begin{example}
Let $M$ be a compact orientable connected Lipschitz $n$-manifold. Every such manifold admits a finite atlas of bi-Lipschitz maps $\psi_i \colon U_i \to M$ where $U_i\subset \R^{n-1}\times [0,\infty)$ are open and the a.e.\ defined differentials of the coordinate transitions $\psi_j^{-1}\circ \psi_i$ are orientation preserving. By choosing a subordinate Lipschitz partition of unity and defining it locally in the charts, one can as in the smooth case integrate Lipschitz differential forms $h \, d\pi_1 \wedge \dots \wedge d\pi_n$. In particular one obtains an (up to sign) uniquely defined fundamental integer-rectifiable current $\bb{M}\in \mathcal{I}_n(M)$. Furthermore, by the Lipschitz version of Stokes' theorem, the manifold boundary of $M$ coincides with the current boundary of $\bb{M}$. That is $\partial \bb{M}= \bb{\partial M}$ and hence $\bb{M}\in \bI_k(M)$. Finally, since $\norm{\bb{M}}=\mu^{m*}_M$, and $M$ is locally bi-Lipschitz equivalent to an open set in $\R^{n-1}\times [0,\infty)$, we deduce that $\set  \bb{M}=M$. Hence, $(M,\bb{M})$ is an integral current space.
\end{example}

Every convex body \(C\subset \R^n\) is a compact connected oriented Lipschitz \(n\)-manifold. In particular, \(\bb{C}\) is an integral \(n\)-current,  and so the term \(\bb{\partial C}\) appearing in Theorem~\ref{thm:rigidity-result-current-version} is a well-defined integral \((n-1)\)-current and satisfies \(\bb{\partial C}= \partial \bb{C}\). 

\section{Mass preserving \(1\)-Lipschitz maps}\label{sec:3}
In this section we make some simple general observations concerning mass preserving $1$-Lipschitz maps. The first of these is the following.
\begin{lemma}
\label{lem:mass-of-preimages}
Let $X,Y$ be complete metric spaces, $f\colon X\to Y$ be $1$-Lipschitz and $T\in \mathcal{R}_n(X)$. If $\mass(T)\leq \mass(f_\# T)$ then 
\begin{equation}
\label{eq:mass-of-preimages}
    \norm{f_\# T}(A)=\norm{T}(f^{-1}(A))
\end{equation}
for every Borel set $A\subset Y$. Furthermore,
\begin{equation}\label{eq:second-part}
f(\spt T)\subset \spt f_{\#}T \quad \textnormal{and}\quad \norm{f_{\#}T}(Y\setminus f(\set T))=0.
\end{equation}
\end{lemma}
\begin{proof}
The inequality $\norm{f_\# T}(A)\leq \norm{T}(f^{-1}(A))$ is readily implied by the map $f$ being $1$-Lipschitz, the characterization of the mass measure given in \cite[Proposition~2.7]{amb00} and the definition of the push-forward. 
Applying this inequality to $A$ and $Y\setminus A$, we obtain
\begin{align*}
\mass(f_\# T) &=\norm{f_\# T}(A)+ \norm{f_\# T}(Y \setminus A)\\
&\leq \norm{T}(f^{-1}(A))+ \norm{T}(X \setminus f^{-1}(A))=\mass (T).
\end{align*}
By our assumption $\mass (T)\leq \mass (f_\# T)$, this inequality chain must be rigid and hence \eqref{eq:mass-of-preimages} follows. 

To finish the proof, we show \eqref{eq:second-part}. If $y=f(x)$ with $x\in \spt T$ and $U$ is an open neighbourhood of $y$ then $f^{-1}(U)$ is an open neighbourhood of $x$ and hence \[
\norm{f_{\#}T}(U)=\norm{T}(f^{-1}(U)>0.
\]
In particular, it follows that $f(\spt T)\subset \spt f_{\#}T$. Finally,
\[\norm{f_{\#}T}(Y\setminus f(\set T))=\norm{T}(X\setminus \set T)=0\]
which completes the proof.
\end{proof}

\begin{lemma}
\label{lem:mass-of-images}
Let $X$ be a complete metric space, $f\colon X\to \mathbb{R}^N$ be $1$-Lipschitz and $T\in \mathcal{I}_n(X)$ be such that $f_\# T=\bb{M}$ where $M\subset \R^N$ is a compact $n$-dimensional Lipschitz manifold. If further $\mass (T)\leq \mathcal{H}^n(M)$, then for every Borel set $A\subset X$, it follows that \(f(A\cap \set T)\) is \(\mathcal{H}^n\)-measurable with
\begin{equation}
    \mathcal{H}^n(f(A\cap \set T))=\norm{T}(A)
\end{equation}
and 
\begin{equation}
    \mathcal{H}^0(f^{-1}(p)\cap \set T)=1
\end{equation}
for $\mathcal{H}^n$-almost every $p\in M$.
\end{lemma}

Naively one might hope that $f$ preserves the mass of all Borel subsets $A\subset X$ even in the more general setting of Lemma~\ref{lem:mass-of-preimages}. There are however two obstacles. The more obvious one is that multiplicities might add up. This is excluded here by assuming $T$ to be integral and that the push-forward current has 'multiplicity one'. The more subtle one is that $\spt T\setminus \set T$ is always a $\norm{T}$-nullset but might in general have positive $\mathcal{H}^n$-measure (see Example~\ref{ex:spt-vs-set}). In particular we cannot exclude the possibillity that the image of this set does have positive $\mathcal{H}^n$-measure.

\begin{proof}[Proof of Lemma~\ref{lem:mass-of-images}]
We may suppose that \(X= \spt T\). In particular, it then follows from Lemma~\ref{lem:mass-of-preimages} that \(f(X)\subset M\). Since \(T\in \mathcal{I}_n(X)\), there are Borel sets $B_i\subset \mathbb{R}^n$, bi-Lipschitz embeddings $\varphi_i \colon B_i\to X$ and Borel functions $\Theta_i \colon B_i \to \mathbb{Z}\setminus\{0\}$ such that \(\set T\) and $S=\bigcup_{i\in \mathbb{N}}\varphi_i(B_i)$ are equal up to a set of \(\mathcal{H}^n\)-measure zero, 
\begin{equation}\label{eq:decomposition-1}
    T=\sum_{i\in \mathbb{N}} \varphi_{i\#} \bb{\Theta_i} \quad \textnormal{and} \quad \norm{T}=\sum_{i\in \mathbb{N}}\, \norm{\varphi_{i\#} \bb{\Theta_i}}.
\end{equation}
 Since $\mathcal{H}^0$ is the counting measure, one has for every Borel set $A\subset X$ that
\begin{align}\label{eq:x}
    \mathcal{H}^n(f(A\cap S))&\leq \int_{\R^N} \mathcal{H}^0\left( f^{-1}(p)\cap A \cap S \right)\, \mathrm{d}\mathcal{H}^n(p) \nonumber \\
    &\leq\sum_{i\in \mathbb{N}}\int_{M} \mathcal{H}^0\left( f^{-1}(p)\cap A \cap \varphi_i(B_i)\right)\, \mathrm{d}\mathcal{H}^n(p).
\end{align}
Moreover, using the area formula, that the metric differentials $\md (f\circ \varphi_i)_q$ are almost everywhere Euclidean and the monotonicity of Jacobians, we get 
\begin{align}\label{eq:xx}
 \int_{M}\mathcal{H}^0\left( f^{-1}(p)\cap A \cap \varphi_i(B_i)\right)\, \mathrm{d}\mathcal{H}^n(p)&\leq \int_{\varphi^{-1}_i(A)\,\cap \,B_i}  \Jac^{b}(\md (f\circ \varphi_i)_q)\ \textrm{d}q \nonumber \\
 &=\int_{\varphi^{-1}_i(A)\,\cap \,B_i}  \Jac^{m*}(\md (f\circ \varphi_i)_q)\ \textrm{d}q \nonumber\\
    &\leq  \int_{\varphi^{-1}_i(A)\,\cap \,B_i}  \Jac^{m*}(\md (\varphi_i)_q) \ \textrm{d}q 
\end{align} 
for every \(i \in \N\). Therefore, using that \( \abs{\Theta_i(q)} \geq 1\) for all \(q\in B_i\), we arrive at
\begin{align} \label{eq:measure_of_images}
    \mathcal{H}^n(f(A\cap S))&\leq \sum_{i\in \mathbb{N}} \int_{\varphi^{-1}_i(A)\,\cap \,B_i}  |\Theta_i(x)|\cdot \Jac^{m*}(\md (\varphi_i)_q)\ \textrm{d}q \nonumber\\
        &=\sum_{i\in \mathbb{N}} \,\norm{\varphi_{i\#}\bb{\Theta_i}}(A) 
    =\norm{T}(A),
\end{align}
where in the last equality we used \eqref{eq:decomposition-1}. 
Since $f$ is Lipschitz, and \(\set T\) and $S$ agree up to $\mathcal{H}^n$-nullsets, we have that
\[
\mathcal{H}^n(f(A\cap \set T))=\mathcal{H}^n(f(A\cap S))
\]
and 
\[
\mathcal{H}^0\left( f^{-1}(p)\cap A \cap \set T \right)=\mathcal{H}^0\left( f^{-1}(p)\cap A \cap S\right)
\]
for $\mathcal{H}^n$-almost every $p\in M$. Now, as in the proof of Lemma~\ref{lem:mass-of-preimages},
\begin{align}\label{eq:mass_of_images}
\mathcal{H}^n(f(\set T))&\leq \mathcal{H}^n(f(A\cap \set T))+\mathcal{H}^n(f((X\setminus A)\cap \set T)) \nonumber\\
&\leq \norm{T}(A)+\norm{T}(X\setminus A)=\mass(T).
\end{align}
Lemma~\ref{lem:mass-of-preimages} tells us that $f(\set T)\subset \spt \bb{M}=M$, and so 
\begin{align*}
\mathcal{H}^n(M\setminus f(\set T))&=\norm{\bb{M}}(M\setminus f(\set T))\\
 &=\norm{T}(X\setminus f^{-1}(f(\set T)))\leq \norm{T}(X\setminus \set T)=0.
\end{align*}
In particular, \eqref{eq:mass_of_images} is rigid and hence so are \eqref{eq:measure_of_images},  \eqref{eq:xx} and \eqref{eq:x}. By our previous observations this gives the claimed equalities.
\end{proof}

\section{Proof of Theorem~\ref{thm:rigidity-result-current-version}}\label{sec:4}

\subsection{The case \(n=1\)}\label{sec:case-n=1}
In the following we prove Theorem~\ref{thm:rigidity-result-current-version} for \(n=1\). The general case \(n\geq 2\) is proved in Section~\ref{sec:proof-of-main-thm} by reducing it to this case. In the proof we use metric 1-currents induced by curves. For a Lipschitz curve $\gamma\colon [a,b]\to X$ into a metric space $X$, the integral 1-current $\bb \gamma\coloneqq\gamma_{\#} \bb{[a,b]}$ is given by 
\[
\bb\gamma (h,\pi_1)=\int_a^b h(\gamma(t))(\pi_1\circ\gamma)'(t)\,\ud t,\quad (h,\pi_1)\in \mathcal D^1(X).
\]
Note that the boundary of $\bb\gamma$ is given by $\partial\bb\gamma(h)=h(\gamma(b))-h(\gamma(a))$, for all $h\in \mathcal D^0(X)$. If \(\gamma\) is a loop and $\gamma|_{[a,b)}$ is injective, we say that $\gamma$ is a \emph{simple} Lipschitz loop. By \cite[Theorem 5.3]{Bonicatto-Del-Nin-Pasqualetto-2022}, the integral \(1\)-current \(T=\bb{X}\) admits a decomposition
\[
T=\sum_{i\in I} \,\bb{\gamma_i}+ \sum_{j\in J}\,\bb{\eta_j},
\]
where $I, J$ are countable index sets, each \(\gamma_i\) is an injective Lipschitz curve, each $\eta_j$ is a simple Lipschitz loop,
\begin{equation}\label{eq:sum-of-total-mass}
\mass(T)=\sum_{i \in I}\mass(\bb{\gamma_i}) +\sum_{j \in J}\mass(\bb{\eta_j})=\sum_{i \in I}\ell(\gamma_i) +\sum_{j \in J}\ell(\eta_j)
\end{equation}
and
\begin{equation}\label{eq:sum-of-total-boundary-mass}
\mass(\partial T)=\sum_{i\in I} \mass(\partial\bb{\gamma_i})+\sum_{j\in J}\mass(\partial\bb{\eta_j})=2|I|.
\end{equation}
By assumption \[
\mass (\partial T)=\mass(f_\#(\partial \bb{B^1}))=\mass(\bb{1}-\bb{-1})=2
\]
and hence \eqref{eq:sum-of-total-boundary-mass} implies $|I|=1$. Henceforth we will denote the unique injective curve $\gamma_j$ by $\gamma\colon [a, b]\to X$, and by $x_1,x_2$ the endpoints of $\gamma$. Since $\partial T=\partial \bb{\gamma}=\bb{x_2}-\bb{x_1}$ 
and $f_\#(\partial T)=\bb{1}-\bb{-1}$ we conclude that $f(x_2)=1$ and $f(x_1)=-1$. In particular, since $f$ is $1$-Lipschitz,
\begin{equation}
\label{eq:length-estimate}
2\leq d(x_1,x_2)\leq \ell(\gamma)\leq \ell(\gamma)+\sum_{j\in J}\ell(\eta_j)=\mass(T)\leq 2
\end{equation}
This implies that $d(x_1,x_2)=2$, $T=\bb{\gamma}$, and $\gamma$ is a geodesic connecting $x_1$ to $x_2$. Since $X=\set T=\spt T=\gamma([a, b])$ 
we conclude that $X$ is isometric to $B^1$. In particular, because $X$ is connected and $1,-1\in f(X)$, it follows that $f$ must be surjective. Since $f\colon X\to B^1$ is a surjective $1$-Lipschitz map, 
we conclude that $f$ is an isometry. 
\subsection{From slice-isometry to isometry}\label{sec:4.2}
The aim of this section is to prove the following proposition which shows that, to obtain Theorem~\ref{thm:rigidity-result-current-version}, it suffices to prove that $f$ is an isometry when restricted to certain slices. 
\begin{proposition}\label{lemm:step-1}
Let \(n\geq 2\),  \(X\) be an integral $n$-current space, $C\subset \mathbb{R}^n$ be a convex body and \(f\colon X\to \mathbb{R}^n\) be a \(1\)-Lipschitz map
such that \(f_\# \bb{X}=\bb{C}\) and \(\mass(\bb{X})\leq \mathcal{H}^n(C)\). Further, suppose that \(k\in \{1, \dots, n-1\}\) and for every orthogonal projection $\rho\colon  \mathbb{R}^n\to \mathbb{R}^{k}$ the following holds: \textnormal{For $\mathcal{H}^k$-almost every $p\in \R^{k}$ the restriction of $f$ is an isometry $\spt \langle T ,\rho \circ f,p \rangle \to \spt \langle \bb{C} ,\rho,p \rangle$.} Then $f$ is an isometry $X\to C$.
\end{proposition}

Here, we use the convention that $\rho\colon  \mathbb{R}^n\to \mathbb{R}^{k}$ is called \emph{orthogonal projection} if there are a \(k\)-plane \(H\susbet \R^n\) and an isometry \(\phi\colon \R^k \to H\), such that \(\phi \circ \rho\) is equal to the orthogonal projection \(\R^n \to H\).

To prove Theorem \ref{thm:rigidity-result-current-version} we will apply Proposition \ref{lemm:step-1} with $k=n-1$ to reduce it to the $n=1$ case handled in the previous subsection. Another natural option would be to take $k=1$ reducing the theorem to the $n-1$ case and performing an induction argument. 
For the proof of Proposition~\ref{lemm:step-1} we need the following simple consequence of the Lebesgue density theorem and Fubini's theorem. 
\begin{lemma}
\label{cor:projection-from-dense-set}
Let $n,k \in \mathbb{N}$ with $k<n$ and $A_1,A_2\subset \mathbb{R}^n$ be $\mathcal{H}^n$-measurable subsets such that $\mathcal{H}^n(A_i)>0$. Then there exists an orthogonal projection $\rho \colon \mathbb{R}^n\to \mathbb{R}^{k}$ and an $\mathcal{H}^{k}$-measurable $E\subset \rho(A_1)\cap \rho(A_2)$ with $\mathcal{H}^{k}(E)>0$ such that for every $p\in E$ the respective sections $\rho^{-1}(p)\cap A_i$ are $\mathcal{H}^{n-k}$-measurable with $\mathcal{H}^{n-k}(\rho^{-1}(p)\cap A_i)>0$.
\end{lemma}
\begin{proof}
Let $p_i\in A_i$ be Lebesgue density points, i.e. $\Theta_n(A_i,p_i)=1$, and set $v\coloneqq p_1-p_2$. For $F\coloneqq A_1\cap (A_2+v)$, we claim that $\mathcal{H}^n(F)>0$. Indeed if $\mathcal{H}^n(F)=0$ we arrive at the following contradiction:
\[\Theta_n(A_1\cup (A_2+v),p_1)=\Theta_n(A_1,p_1)+\Theta_n(A_2,p_2)=2>1.\] 
Now we choose an orthogonal projection $\rho \colon \mathbb{R}^{n}\to \R^{k}$ with $\rho(v)=0$. Then \[\rho(F)\subset \rho(A_1)\cap \rho(A_2+v)=\rho(A_1)\cap\rho(A_2).\] Since $F$ is $\mathcal{H}^n$-measurable with $\mathcal{H}^n(F)>0$, Fubini's theorem implies that there is an $\mathcal{H}^{k}$-measurable set $E\subset \rho(F)$ with $\mathcal{H}^{k}(E)>0$ such that for every $p\in E$ the section $\rho^{-1}(p)\cap F$ is $\mathcal{H}^{n-k}$-measurable with $\mathcal{H}^{n-k}(\rho^{-1}(p)\cap F)>0$. Since $F\subset A_1$, $F\subset A_2+v$ and $\rho^{-1}(p)+v=\rho^{-1}(p)$ we have $\mathcal{H}^{n-k}(\rho^{-1}(p)\cap A_i)>0$ for \(i=1,2\). Finally by Fubini $\rho^{-1}(p)\cap A_i$ is $\mathcal{H}^{n-k}$-measurable for almost every $p\in \R^{k}$ and hence we may also assume that $\rho^{-1}(p)\cap A_i$ is $\mathcal{H}^{n-k}$-measurable for every $p\in E$.
\end{proof}
\begin{proof}[Proof of Proposition~\ref{lemm:step-1}]
Let $x_1,x_2 \in X$ and $\delta>0$. Then, the balls $B(x_i,\delta)$ are  $\norm{T}$-measurable with $\norm{T}(B(x_i,\delta))>0$. By Lemma~\ref{lem:mass-of-images}, setting $B_i \coloneqq B(x_i,\delta)\cap \set T$, the sets $L_i \coloneqq f(B_i)$ are $\mathcal{H}^n$-measurable with $\mathcal{H}^n(L_i)>0$.

By Lemma~\ref{cor:projection-from-dense-set}, there are an orthogonal projection $\rho \colon \mathbb{R}^n\to\mathbb{R}^{k}$ and a measurable $E\subset \rho(L_1)\cap \rho(L_2)$ with $\mathcal{H}^{k}(E)>0$ such that $\rho^{-1}(p)\cap L_i$ is of positive $\mathcal{H}^{n-k}$-measure for every $p\in E$. By Lemma~\ref{lem:support-of-sliced-measure-equals-slice} we may further assume that for every $p\in E$, 
\begin{equation}
\label{eq:proof-step-1}
\set \langle T, \widehat{\rho},p\rangle = \widehat{\rho}^{\hspace{0.2em}-1}(p)\cap \set T
\end{equation}
up to an $\mathcal{H}^{n-k}$ null set and by our assumption that for every $p\in E$ the restriction of $f$ defines an isometry $\spt \langle T,\widehat{\rho}, p\rangle \to \spt \langle \bb{C},\rho,p\rangle$. 

Now let $p\in E$. Then for each $i$ the slice $\rho^{-1}(p)\cap L_i$ is of positive $\mathcal{H}^{n-k}$-measure. Since $f$ is $1$-Lipschitz and $f(\widehat{\rho}^{\hspace{0.2em}-1}(p)\cap B_i)=\rho^{-1}(p)\cap L_i$, this implies that also $\widehat{\rho}^{\hspace{0.2em}-1}(p)\cap B_i$ is of positive $\mathcal{H}^{n-k}$-measure. Thus we deduce from \eqref{eq:proof-step-1} that \[\mathcal{H}^{n-k}(B_i \cap \set \langle T, \widehat{\rho},p\rangle)= \mathcal{H}^{n-k}(\widehat{\rho}^{\hspace{0.2em}-1}(p)\cap B_i)>0.
\]
In particular, we may respectively choose points $y_i\in B_i\cap \spt \langle T, \widehat{\rho},p\rangle$. Since $f$ is an isometric embedding on $\set \langle T, \widehat{\rho},p\rangle$, we have that 
\[
d(y_1,y_2)=|f(y_1)-f(y_2)|.
\]
Since $y_i\in B(x_i,\delta)$ and $f$ is continuous, by letting $\delta\to 0$ we conclude that $d(x_1,x_2)=|f(x_1)-f(x_2)|$. In particular, since $x_1, x_2\in X$ were arbitrary, $f$ defines an isometric embedding. By Lemma~\ref{lem:mass-of-preimages}, $f(X)$ is dense in $C$, and so it follows that \(\bar{f} \colon \overline{X} \to C\) is an isometry. Since \(f_\# \bb{X}=\bb{C}\), we find that $X=\set T=\spt T=\overline{X}$ and hence the claim follows.
\end{proof}

\subsection{Proof of Theorem \ref{thm:rigidity-result-current-version}}\label{sec:proof-of-main-thm}
In the following, we suppose that \(n\geq 2\). The case when \(n=1\) is treated in Section~\ref{sec:case-n=1}. To prove the theorem it suffices to show that the assumptions of Proposition~\ref{lemm:step-1} are satisfied for $k=n-1$. So let $\rho \colon \mathbb{R}^n\to \mathbb{R}^{n-1}$ be an orthogonal projection.
Letting $T_p=\langle \bb{X}, \widehat{\rho},p \rangle$, \(X_p= \set T_p\), and $C_p=C\cap \rho^{-1}(p)$, 
we claim that the following conditions are satisfied for \(\Ha^{n-1}\)-almost every $p\in\mathbb{R}^{n-1}$:
\begin{enumerate}
    \item \((X_p, T_p)\) is an integral current space
    \item $f_\# \bb{X_p}=  \bb{C_p}$. 
    \item $\mass(\bb{X_p})\leq \mathcal{H}^1(C_p)$.
    \item $\mass(\bb{\partial X_p})\leq \mathcal{H}^0(\partial C_p)$.
\end{enumerate}

Condition (1) follows directly from the properties of the slicing operator discussed in Section~\ref{sec:slicing}.

 Applying the constancy theorem (see \cite[Corollary 3.13]{fed60}) to the integral \(n\)-cycle \(T=f_\# \bb{X}-\bb{C}\), it follows that \(T=0\) and thus $f_\# \bb{X}=\bb{C}$. Alternatively, this can be seen by applying the deformation theorem (see e.g. \cite[Theorem A.2]{basso2021undistorted}). Hence, using \eqref{eq:slice-and-push-forward-commute}, we find that 
 \[
 f_\# \bb{X_p}=f_\# \langle \bb{X}, \widehat{\rho},p \rangle=\langle f_\# \bb{X}, \rho,p \rangle=\langle \bb{C}, \rho,p \rangle
 \]
 for \(\Ha^{n-1}\)-almost every $p$. Notice that \(\bb{C_p}\in \mass_1(\R^n)\) are concentrated on \(\rho^{-1}(p)\), satisfy \eqref{eq:universal-property}, and \(\int_{\R^{n-1}} \mass(\bb{C_p}) \, dp <\infty\). Hence, as these properties uniquely determine the slices \(\langle \bb{C}, \rho,p \rangle\), it follows that \(\langle \bb{C}, \rho,p \rangle=\bb{C_p}\) for \(\Ha^{n-1}\)-almost every $p$, and thus, by the above (2) follows. 
 We proceed by showing (3). By (2), it follows that $\mass(\bb{X_p})\geq \mathcal{H}^1(C_p)$. Hence, using Fubini and \eqref{eq:mass-of-slices}, we find that
\begin{align*}
    \mathcal{H}^n(C)=\int_{\R^{n-1}}\mathcal{H}^{1}(C_p) \ \mathrm{d}p\leq \int_{\R^{n-1}}\mass(\bb{X_p}) \ \mathrm{d}p =\mass(\bb{X} \on (1, \widehat{\rho}))\leq \mass(\bb{X}).
\end{align*}
By our assumption \(\mass(\bb{X})\leq \Ha^{n}(C)\) this equality is rigid, and so (3) follows.

Finally, we prove (4). By Lemma~\ref{lem:mass-of-preimages}, $f(\spt \bb{\partial X})\subset \partial C$, and hence by Lemma~\ref{lem:mass-of-images} for \(\Ha^{n-1}\)-almost every $p$, 
\[
\mathcal{H}^0(\widehat{\rho}^{\hspace{0.2em}-1}(p)\cap \set \partial T)\leq 2.
\]
However, for \(\Ha^{n-1}\)-almost every $p$ we also have by Lemma~\ref{lem:support-of-sliced-measure-equals-slice} that \[\widehat{\rho}^{\hspace{0.2em}-1}(p)\cap \set \partial T =\set \langle \partial T, \widehat{\rho},p \rangle =\set \bb{ \partial X_p}=\spt \bb{ \partial X_p}
\]
Together with \(f_\# \bb{\partial X_p}=\bb{\partial C_p}\) this implies (4). 

Now, since Theorem~\ref{thm:rigidity-result-current-version} is valid when \(n=1\), the restriction of \(f\) is an isometry $\spt \langle T ,\widehat{\rho}, p \rangle \to \spt \langle \bb{C} ,\rho,p \rangle$ for \(\Ha^{n-1}\)-almost every \(p\in \R^{n-1}\). Therefore, as \(\rho\) was arbitrary, Proposition~\ref{lemm:step-1} tells us that \(f\) is an isometry, as desired. \qed


\section{Proof of Corollary~\ref{cor:sphere-rigidity}}\label{sec:5}
The \emph{Euclidean cone} $CX$ over a metric space $X=(X,d)$ is the metric space obtained when endowing $X\times [0,1]$ with the pseudometric
\begin{equation}
    d_C((x,r),(y,s)):=\begin{cases}
    \sqrt{r^2+s^2-2rs\cos\left(d(x,y)\right)} &\textnormal{ if }d(x,y)<\pi,\\
    r+s &\textnormal{ otherwise,}
    \end{cases}
\end{equation}
which defines a metric on the quotient space $CX=X\times [0,1]/\sim$, where $(x,0)\sim (y,0)$ for all $x,y\in X$. Compare also \cite[Section 3.6]{burago2022course}. Observe in particular that the Euclidean cone over the round sphere $S^n$ is isometric to the flat disk $B^{n+1}$. We denote by $H\colon X\times [0,1]\to CX$ and $e\colon X \to CX$ the Lipschitz maps given by $h(x,t)=[(x,t)]$ and $e(x)=[(x,1)]$. It is a consequence of the monotonicity of the cosine function on $[0,\pi]$ that $e$ is $1$-Lipschitz. For the same reason, also if $f\colon X\to Y$ is $1$-Lipschitz then the map $Cf\colon CX \to CY$ defined by \([(x,r)]\mapsto [(f(x), r)]\) is $1$-Lipschitz as well. 

For any $T\in \bI_n(X)$, we set \(CT\coloneqq H_\#( T\times \bb{0,1})\), where the product current \(T\times \bb{0,1}\in \bI_{n+1}(X\times [0,1])\) is defined as in \cite[Section 3.3]{basso2021undistorted}. By construction, $CT\in \bI_{n+1}(CX)$, and one has $\set CT=H(\set T\times [0,1])$, $\spt CT=H(\spt T\times [0,1])$, \(\partial (CT)= C(\partial T)+e_\# T\) and \((Cf)_\#CT=C(f_\#T)\). Moreover, if $T$ is represented as in \eqref{eq:rep-rect-current} by functions $\Theta_i\in L^{1}(\mathbb{R}^n,\mathbb{Z})$ and bi-Lipschitz embeddings $\varphi_i\colon K_i \to X$ then setting $\widetilde{\Theta}_i(x,t):=\Theta(x)$, $\widetilde{K}_i:=K_i\times[0,1]$ and $\widetilde{\varphi}_i(x,t):=[(\varphi_i(x),t)]$ we find that
\begin{equation}
\label{eq:cone-decomp}
    CT=\sum_{i\in \N} \widetilde{\varphi}_{i\#} \bb{\widetilde{\Theta}_i} \ \ \textnormal{and} \ \ \mass(CT)=\sum_{i\in \N} \mass (\widetilde{\varphi}_{i\#} \bb{\widetilde{\Theta}_i}).
\end{equation}

The following Lemma shows that the $\mass^{\ir}$-mass of $CT$ is analogous to the volume of cones in Euclidean space.
\begin{lemma}
\label{lem:coning-ineq}
If $T\in \bI_n(X)$ then
\begin{equation}
\label{eq:coning-ineq}
    \mass^{\ir}(CT)=\frac{1}{n+1}\cdot \mass^{\ir}(T).
\end{equation}
\end{lemma}
Note that for $\mass$, instead of \eqref{eq:coning-ineq}, only a weaker inequality without the factor $\frac{1}{k+1}$ holds, compare \cite[Lemma~3.5]{basso2021undistorted}. For this reason we can prove Corollary~\ref{cor:sphere-rigidity} only for $\mass^{\ir}$.
\begin{proof}
For a Lipschitz map $\varphi \colon K \to X$ with $K\subset \mathbb{R}^k$ we consider the corresponding map $\widetilde{\varphi}\colon K\times [0,1]\to CX$ as above. Then then for almost every $(x,r)\in K\times [0,1]$ one has for every $(v,s)\in \mathbb{R}^n\times\mathbb{R}=\mathbb{R}^{n+1}$ that 
\begin{equation}
    \big(\hspace{-0.2em}\md \widetilde{\varphi}_{(x,r)}(v,s)\big)^2
    =\lim_{\varepsilon\downarrow 0}  \frac{(r+\varepsilon s)^2+r^2-2r(r+\varepsilon s)\cos\big( d\big(\varphi(x+\varepsilon v),\varphi(x)\big)\big)}{\varepsilon^2}.
\end{equation}
Using that \(1-\cos(x)=\frac{x^2}{2}+O(x^4)\) we deduce that
\[
 \md \widetilde{\varphi}_{(x,r)}(v,s)=\sqrt{r^2\cdot \left(\md \varphi_x(v)\right)^2+s^2}.
\]
Thus by \eqref{eq:ir-product} one has 
\begin{equation}
    \Jac^{\ir}(\md \widetilde{\varphi}_{(x,r)})=\Jac^{\ir}(r\cdot \md \varphi_x)=r^{n}\cdot\Jac^{\ir}(\md \varphi_x).
\end{equation}
Using this observation, the charts $\widetilde{\varphi}_i$ as in \eqref{eq:cone-decomp}, and Fubini we obtain
\begin{equation}
    \mass^{\ir}(CT)=\int_0^1 r^n\ \textrm{d}r\cdot \mass^{\ir}(T)=\frac{1}{n+1}\cdot \mass^{\ir}(T)
\end{equation}
as desired.
\end{proof}

\begin{proof}[Proof of Corollary \ref{cor:sphere-rigidity}]
Let $T=\bb{X}$. Then $CX=(CX,CT)$ is an integral current space and $Cf\colon CX\to B^{n+1}$ is a $1$-Lipschitz map with $(Cf)_{\#} CT=\bb{B^{n+1}}$ (see the discussion before Lemma \ref{lem:coning-ineq}). Furthermore by Lemma~\ref{lem:coning-ineq}
\begin{equation}
    \mass(CT)\leq \mass^{\ir}(CT)=\frac{1}{n+1}\cdot \mass^{\ir}(T)\leq \frac{1}{n+1}\cdot \vol^{n}(S^n)=\vol^{n+1}(B^{n+1})
\end{equation}
and
\begin{equation}
    \mass(\partial(CT))=\mass(e_{\#}T)\leq \mass(T)\leq \mass^{\ir}(T)\leq \vol^{n}(S^n).
\end{equation}
Hence Theorem~\ref{thm:rigidity-result-current-version} implies that $Cf\colon CX\to B^{n+1}$ is an isometry.

Now if $x,y\in X$ are such that $d_{CX}([(x,1)],[(y,1)])=d_{B^{n+1}}(f(x),f(y))<2$ then $d(x,y)<\pi$ and
\begin{equation}
    \sqrt{2-2\cos(d(x,y))}=\sqrt{2-2\cos\left(d_{S^n}(f(x),f(y))\right)}.
\end{equation}
Since $\cos$ is injective on $[0,\pi]$ this implies that $d(x,y)=d_{S^n}(f(x),f(y))$.

To prove the equality also for $x,y\in X$ with $d_{B^{n+1}}(f(x),f(y))=2$ we choose $z\in X$ with $x\neq z\neq y$. Since $f$ is bijective and $f(z)$ lies on a $S^{n}$-geodesic from $f(x)$ to $f(y)$, the previous case gives
\begin{equation}
    d(x,y)\leq d(x,z)+d(z,y)=d_{S^n}(f(x),f(z))+d_{S^n}(f(z),f(y))=d_{S^n}(f(x),f(y)).
\end{equation}
Since $f$ is $1$-Lipschitz, this implies the claim.
\end{proof}

\section{Proof of Theorem~\ref{thm:main-i}}\label{sec:6}

\subsection{The lower bound}\label{sec:6.1}
We start by proving inequality \eqref{eq:main-ineq}. It is readily implied by the following lemma since $\mass \leq \mass^{\ir}$.

\begin{lemma}\label{lem:absolute-filling-vol}
Let \(C\susbet \R^n\) be a convex body and suppose \(\iota\colon\partial C \to X\) is an isometric embedding into an integral current space \(X\) such that \(\iota_\# \bb{\partial C}=\bb{\partial X}\). Then
\[
\mass(\bb{X})\geq \vol^n(C).
\]
\end{lemma}
The following proof is essentially based on an observation due to Gromov \cite[Proposition~2.1.A]{gromov-1983}.
\begin{proof}
Since \(\iota\) is an isometric embedding and \(\spt \iota_\# \bb{\partial C} \susbet \overline{\iota(\partial C)}\), we can conclude that \(j_\# \bb{\partial X}=\bb{\partial C}\), where \(j\colon \partial X \to \partial C\) denotes the inverse of \(\iota\). Let \(h\colon \R^n \to \ell_\infty^n\) be the identity map, which is \(1\)-Lipschitz. Obviously, \(h\circ j\) is also \(1\)-Lipschitz, and since \(\ell_\infty^n\) is injective, there exists a \(1\)-Lipschitz extension \(f\colon X\to \ell_\infty^n\) of \(h \circ j\). Since \(\R^n\) and \(\ell_\infty^n\) are bi-Lipschitz equivalent, it follows directly from the constancy theorem (see \cite[Corollary 3.13]{fed60}) that the \(n\)-cycle \(T=f_\# \bb{X}-h_\# \bb{C}\) is equal to the zero current and thus \(f_\# \bb{X}=h_\# \bb{C}\). In particular, \(\mass(\bb{X}) \geq \mass(h_\# \bb{C})\). But 
\[
\mass(h_\# \bb{C})=\mu^{m\ast}(h(C))=\mu^{m\ast}(C)=\mass(\bb{C}),
\]
where in the second equality we have used that \(\Jac^{m\ast}(\norm{\cdot}_\infty)=1\). 
\end{proof}
\subsection{Rigidity}\label{sec:6.2}

In this subsection we use the techniques of Burago and Ivanov developed in \cite{bi10, bi13} to deduce the rigidity statement in Theorem~\ref{thm:main-i} from the Lipschitz-volume rigidity result Theorem~\ref{thm:rigidity-result-current-version}.

For the proof we need the following auxiliary spaces: We denote by $\Ell\coloneqq L^\infty(S^{n-1})$ the Banach space of (equivalence classes of) Borel measurable essentially bounded functions $S^{n-1}\to \R$ endowed with the usual norm~$\norm{\cdot}_\infty$. Furthermore we consider the space $\Ell_2\coloneqq L^2(S^{n-1})$ equipped with the inner product
\[
\langle f, \, g\rangle_{2} \coloneqq \frac{n}{\vol^{n-1}(S^{n-1})}\int_{S^{n-1}}fg\ \ud\Ha^{n-1}
\]
and the corresponding norm $\|f\|_2\coloneqq \sqrt{\langle f,f\rangle_{2}}$. The need for this particular normalization constant will become clear below. The properties of these spaces relevant for the proof of Theorem~\ref{thm:main-i} are summarized in the following lemma.

\begin{lemma}
\label{lem:proof-of-main-thm}
The following hold true:
\begin{itemize}
    \item[(1)] $\Ell$ is injective. 
    \item[(2)] $\Ell_2$ is a Hilbert space.
    \item[(3)] The canonical embedding $I\colon \Ell \to \Ell_2$ is Lipschitz and
\begin{equation}\label{eq:volume-non-increasing}
\norm{I_\# T}^{\ir} \leq I_\# \norm{T}^{\ir}
\end{equation}
for every $T\in \mathcal{R}_n(\Ell)$.
\item[(4)] There is a linear map $\Phi \colon \mathbb{R}^n\to \Ell$ such that $\Phi \colon \mathbb{R}^n\to \Ell$ and the composition $I\circ \Phi \colon \R^n \to \Ell_2$ are both isometric embeddings.
\end{itemize}
\end{lemma}

\begin{proof}
To prove (1), it suffices to combine McShane's extension theorem with \cite[Lemma~5.1]{bi10}. Moreover, (2) holds true since $\norm{\cdot}_2$ is just a rescaling of the usual $L^2$-norm.

A straightforward application of Hölder's inequality shows that $I$ is $\sqrt{n}$-Lipschitz. To complete the proof of (3) it remains to show \eqref{eq:volume-non-increasing}. In light of  \eqref{eq:Jacobian-mass-characterization} and \eqref{eq:jac_volume} it suffices to show that
\begin{align}\label{eq:jac-ineq}
\Jac^{\ir}(\md(I\circ\varphi)_x)\le \Jac^{\ir}(\md\varphi_x)\quad \text{for \(\Ha^n\)-almost every}\ x\in E,
\end{align}
for every bi-Lipschitz map $\varphi:E\to \Ell$ from a Borel set $E\subset \R^n$. If $\bar\varphi:\R^n\to \Ell$ is a Lipschitz extension of $\varphi$, then $\md\varphi_x=\md\bar\varphi_x$ for \(\Ha^n\)-almost every $x\in E$. Thus we may assume that $\varphi$ is defined on $\R^n$. Let $x\in \R^n$ be a point where $\varphi$ admits a metric differential $\md\varphi_x$ and $I\circ\varphi$ admits a Fréchet differential $A_x\coloneqq (I\circ\varphi)'(x):\R^n\to \Ell_2$. Observe that 
\begin{align*}
 \md(I\circ\varphi)_x(v)=\|A_x(v)\|_2,\quad v\in \R^n.
\end{align*}
We first claim that $V_x\coloneqq A_x(\R^n)\subset \Ell$. Indeed, since 
\[
\Big\|\frac{\varphi(x+hv)-\varphi(x)}{h}\Big\|_\infty\le \Lip(\varphi),\ \textrm{ and }\ A_x(v)=\lim_{h\to 0}\frac{\varphi(x+hv)-\varphi(x)}{h}\ \textrm{ in }\Ell_2
\]
 for each $v\in\R^n$ it follows that $A_x(v)\in\Ell$ and, moreover, that 
\begin{align*}
    \left\|A_x(v)\right\|_\infty \le \liminf_{h\to 0}\Big\|\frac{\varphi(x+hv)-\varphi(x)}{h}\Big\|_\infty=\md\varphi_x(v),\quad v\in \R^n.
\end{align*}
We now prove \eqref{eq:jac-ineq}. The claim is trivially true if $A_x$ is not injective. Thus we may assume that $A_x:\R^n\to V_x$ is a linear isomorphism. In particular we have that
\begin{align}\label{eq:jac-id}
\Jac^{\ir}(\md(I\circ\varphi)_x)=|\det (I\circ A_x)|,\quad \Jac^{\ir} (\md\varphi_x)\ge \Jac^{\ir} (s)=|\det{L\inv}|,
\end{align}
where $L\colon\R^n\to \R^n$ is a linear isomorphism such that $L(B^n)$ is the John ellipsoid of the norm $s=\|A_x(\cdot)\|_\infty$. Note that $(A_x\circ L)(B^n)\subset A_x(B_s)=V_x\cap B_{\Ell}$, and thus $A_x\circ L:\R^n\to \Ell$ is 1-Lipschitz. From \cite[Lemma ~6.1]{bi10} it now follows that the composition $I\circ A_x\circ L:\R^n\to \Ell_2$ is area non-increasing, i.e. $|\det(I\circ A_x\circ L)|\le 1$. Thus
\begin{align*}
|\det(I\circ A_x)|=|\det(I\circ A_x\circ L)|\cdot|\det(L\inv)|\le |\det(L\inv)|,
\end{align*}
which by \eqref{eq:jac-id} implies \eqref{eq:jac-ineq}. 

To prove (4) we consider the linear map $\Phi:\R^n\to \Ell$ defined by 
\(\Phi_x(p)=\langle x,p\rangle, \, p\in S^{n-1},\)
for each $x\in\R^n$. Using the Cauchy-Schwarz inequality, it is easy to check that $\Phi$ is an isometric embedding. It remains to show that $I\circ \Phi$ is an isometric embedding as well. This follows from the proof of \cite[Lemma~4.6]{bi10}. Indeed, for all \(x\in \R^n\) of unit norm, one has
\begin{align*}
\norm{I\circ\Phi(x)}_{2}^2=n\fint_{S^{n-1}} \langle x,  p\rangle ^2 \, d\Ha^{n-1}(p)= \fint_{S^{n-1}} \sum_{i=1}^n \langle e_i,  p\rangle ^2 \, d\Ha^{n-1}(p)=1.
\end{align*}
By linearity of $I\circ \Phi$ this completes the proof.
\end{proof}
The proofs of rigidity in \cite{bi10} and \cite{bi13} rely on a rigidity version of \eqref{eq:volume-non-increasing} that does not apply in our current setting. Nevertheless by applying our Lipschitz-volume rigidity theorem twice we are able to avoid  this difficulty and complete the proof of Theorem~\ref{thm:main-i}.
\begin{proof}[Proof of Theorem~\ref{thm:main-i}]
Let \(X\) be an integral current space and \(\iota\colon \partial C \to X\) be an isometric embedding such that \(\iota_\# \bb{\partial C}=\bb{\partial X}\). Inequality \eqref{eq:main-ineq} follows immediately from Lemma \ref{lem:absolute-filling-vol}. It remains to show that if \(\mass^{\ir}(\bb{X})=\vol^n(C)\), then \(\iota\) can be extended to an isometry \(C \to X\).

Clearly, \(\partial X= \iota(\partial C)\). By Lemma~\ref{lem:proof-of-main-thm}(1) the map \(\Phi\circ \iota^{-1}\colon \partial X \to \Ell\) admits a \(1\)-Lipschitz extension \(f\colon X \to \Ell\). Let \(T\coloneqq f_\# \bb{X}\) and $Z$ be the integral current space $(\set(I_\# T), I_\# T)$ endowed with the subspace metric of \(\Ell_2\). By Lemma~\ref{lem:proof-of-main-thm}(3) and the monotonicity of $\mass^{\ir}$ we have that
\begin{equation}\label{eq:masses-inequality-1}
\mass(I_\# T)\leq \mass^{\ir}(I_\# T)\leq \mass^{\ir}(T)\leq  \mass^{\ir}(\bb{X})\leq \vol^n(C).
\end{equation}
Lemma~\ref{lem:proof-of-main-thm}(2) implies that there is a $1$-Lipschitz projection $P\colon \Ell_2\to (I\circ \Phi)(\R^n)$. Notice that \(P_\# \bb{\partial Z}=(I\circ \Phi)_{\#} \bb{\partial C}\). Now, as \(\mass(\bb{\partial Z})=\vol^{n-1}(\partial C)\) and \(\mass(\bb{Z})\leq \vol^n(C)\), Theorem~\ref{thm:rigidity-result-current-version} implies that the restriction of $P$ to $Z$ defines an isometry $Z\to (I\circ \Phi)(C)$. Since $\Ell_2$ is a Hilbert space and hence uniquely geodesic, we conclude that $Z=(I\circ \Phi)(C)$.  In particular, \eqref{eq:masses-inequality-1} is rigid, and so \(\mass^{\ir}(I_\# T)=\mass^{\ir}(T)=\vol^n(C)\). Therefore, by \eqref{eq:volume-non-increasing}, we get that \(\norm{I_\# T}^{\ir}=I_\# \norm{T}^{\ir}\) and so \(I(\spt T)\subset \spt I_\# T\). Thus, since \(\spt I_\# T=(I\circ \Phi)(C)\) and \(\partial T=\Phi_\# \bb{\partial C}\), it follows that \(T=\Phi_\# \bb{C}\). But \(T=f_\# \bb{X}\) and so
\[
\vol^n(C)=\mass(f_\# \bb{X})\leq \mass(\bb{X}) \leq \mass^{\ir}(\bb{X})\leq \vol^n(C).
\]
By Lemma~\ref{lem:mass-of-preimages}, it follows that \(f(X)\susbet \Phi(C)\). Hence, we can apply Theorem~\ref{thm:rigidity-result-current-version} once again and we conclude that $f\colon X\to \Phi(C)$ is an isometry.
\end{proof}

\section{Intrinsic flat convergence and the Perales question}
\label{sec:Intrinsic flat convergence}

\subsection{Intrinsic flat convergence}\label{sec:7.1} 
Given \(T\in \bI_n(X)\), let
\[
\mathcal{F}_X(T)=\inf \big\{ \mass(U)+\mass(V) : T=U+\partial V, \, U\in \bI_{n}(X), \, V\in \bI_{n+1}(X)\big\}
\]
denote the \emph{flat norm} of \(T\). If the ambient space \(X\) is clear form the context we often write \(\mathcal{F}(T)\) instead of \(\mathcal{F}_X(T)\). 
We say that \(T_i\in \bI_n(X)\) \emph{flat converges} to \(T\in \bI_n(X)\), if \(\mathcal{F}(T-T_i)\to 0\) as \(i\to \infty\). Moreover, we say that \(T_i\) \emph{converges weakly} to \(T\) if \(T_i(h, \pi_1, \dots, \pi_n) \to T(h, \pi_1, \dots, \pi_n)\) as \(i\to \infty\) for every \((h, \pi_1, \dots, \pi_n)\in \mathcal{D}^n(X)\). It is readily verified that flat convergence implies weak convergence. Conversely, if \(X\) admits local coning inequalities and \(\sup \mathbf{N}(T_i) < \infty\), then weak convergence also implies flat convergence (see \cite{wenger07}).

In \cite{sw11}, Sormani and Wenger introduced a notion of flat convergence for currents which are not necessarily defined on the same metric space. The \emph{instrinsic flat distance} between two integral current spaces \(X_1\), \(X_2\) of the same dimension is defined as 
\[
d_{\mathcal{F}}(X_1, X_2)= \inf \mathcal{F}_Z\big(\phi_{1\#}\bb{X_1}-\phi_{2\#}\bb{X_2}\big),
\]
where the infimum is taken over all complete metric spaces \(Z\) and all isometric embeddings \(\phi_{i}\) of \(X_i\) into \(Z\). We say that a sequence \(X_i\) of integral current spaces \emph{converges in the intrinsic flat sense} to an integral current space \(X\) if \(d_{\mathcal{F}}(X_i, X)\to 0\) as \(i\to \infty\). The following Arzelà-Ascoli-type theorem is due to Sormani.

\begin{theorem}[see Theorem 6.1 in \cite{sor18}]\label{thm:Sormani-arzela}
Suppose \(X_i\) are integral current spaces converging to the integral current space \(X\) in the intrinsic flat sense. Further, suppose \(f_i\colon X_i \to Y\) are \(L\)-Lipschitz maps to a compact metric space \(Y\). Then there exists a subsequence, also denoted by \(f_{i}\), that converges pointwise to an \(L\)-Lipschitz map \(f\colon X \to Y\). 
\end{theorem}

Here, $f_i$ is said to converge pointwise to $f$ if there exists a separable complete metric space \(Z\), and isometric embeddings \(\phi_i \colon X_i \to Z\), \(\phi \colon X\to Z\) such that \(\phi_{i\#}\bb{X_i}\) flat converges to \(\phi_{\#}\bb{X}\) and, whenever \(x\in X\) and \(x_i\in X_i\) are such that \(\phi_i(x_i)\) converges to \(\phi(x)\), then \(f_i(x_i)\) converges to \(f(x)\). 

The map \(f\colon X\to Y\) will be called a \emph{Sormani limit} of the subsequence \(f_{i}\). We note that for every \(x\in X\) there is always such a sequence \(x_i\in X_i\) as above. This follows directly from the next lemma.

\begin{lemma}\label{lem:points-in-support-are-limits}
Let \(Z\) be a complete metric space and \(T_i\in \bI_n(Z)\) a sequence flat converging to \(T\in \bI_n(Z)\). Then for every \(z\in \spt T\) there is a sequence \(z_i\in \set T_i\) such that \(z_i\to z\) as \(i\to \infty\). 
\end{lemma}

\begin{proof}
The following argument is due to Wenger (see \cite[Proposition 2.2]{wenger-2011}). 
Let \(z\in \spt T\) and \(\varepsilon >0\). Then, using \cite[Proposition 2.7]{amb00}, one can show there exists \((h, \pi_1, \dots, \pi_n)\in \mathcal{D}^n(X)\) such that \(T(h, \pi_1, \dots, \pi_n)\neq 0\) and \(\spt h \subset B(x, \varepsilon)\). Since \(T_i\) flat converges to \(T\), and therefore in particular converges weakly to \(T\), for any \(i\) that is sufficiently large, one has \(T_i(h, \pi_1, \dots, \pi_n)\neq 0\). It now follows from  Definition~\ref{def:current}\eqref{ax:3} that for every such \(i\) there is \(z_i\in \spt T_i \cap \spt h\). Since \(\varepsilon >0\) was arbitrary and $\set T_i$ is dense in $\spt T_i$, a sequence \(z_i\in \set T_i\) such that \(z_i\) converges to \(z\) is now easily constructed. 
\end{proof}
It turns out that the convergence as in Theorem~\ref{thm:Sormani-arzela} is compatible with push-forwards of currents.
\begin{lemma}\label{lem:functorality-arzela-ascoli}
Let \(X_i\) be a sequence of integral current spaces converging in the intrinsic flat sense to an integral current space \(X\). Suppose further that \(f_i \colon X_i \to \R^N\) are uniformly bounded \(L\)-Lipschitz maps and let \(f\colon X\to \R^N\) be the Sormani limit of some subsequence \(f_{i}\). Then \(f_{i\#} \bb{X_{i}}\) flat converges to \(f_\# \bb{X}\). 
\end{lemma}
\begin{proof}
Since \(f_{i\#} \bb{X_{i}}=0\) and \(f_\# \bb{X}=0\) whenever \(N < n\), we may assume in the following that \(n\leq N\). Due to Theorem~\ref{thm:Sormani-arzela}, there exist a separable complete metric space \(Z\) and isometric embeddings \(\phi_i \colon X_i \to Z\) and \(\phi \colon X \to Z\), such that \(\phi_{i\#} \bb{X_i}\) flat converges to \(\phi_\# \bb{X}\) and the sequence \(f_i\) converges pointwise to \(f\) in the sense that \(f_i(x_i)\to f(x)\) as \(i\to \infty\), whenever \(x\in X\) and \(x_i\in X_i\) are such that \(\phi_i(x_i)\) converges to \(\phi(x)\).  To simplify the notation, we write \(T_i=\phi_{i\#}\bb{X_i}\) and \(T=\phi_{\#}\bb{X}\). 

 By McShane's extension theorem \cite[Theorem 1.27]{brudnyi2012} there exist $(\sqrt{N}L)$-Lipschitz maps 
 \(F_i\colon Z\to \R^N\) and \(F\colon Z\to \R^N\) such that \(f_i=F_i \circ \phi_i\) for all \(i\in \N\) and \(f=F\circ \phi\). In particular, \(F_{i\#} T_i=f_{i\#} \bb{X_i}\) and \(F_\# T=f_\# \bb{X}\). Thus, using the triangle inequality, we get
\begin{equation}\label{eq:triangle-ineq-1}
\mathcal{F}(f_{i\#} \bb{X_i}-f_\# \bb{X}) \leq \mathcal{F}(F_{i\#} T_i-F_{i\#}T)+\mathcal{F}(F_{i\#}T-F_\# T)
\end{equation}
for all \(i\in \N\). Since the maps \(F_i\) are uniformly Lipschitz and \(T_i\) flat converges to \(T\), it follows from \eqref{eq:mass-pf-inequality} that \(\mathcal{F}(F_{i\#} T_i-F_{i\#}T)\to 0\) as \(i\to \infty\).

Next, we show that the other term on the right-hand side of \eqref{eq:triangle-ineq-1} also converges to zero. Notice that \(F_i\circ \phi\) converges pointwise to \(F\circ \phi\). Indeed, let \(x\in X\) and let \(x_i\in X_i\) be a sequence such that \(\phi_i(x_i)\) converges to \(\phi(x)\). The existence of such a sequence is guaranteed by Lemma~\ref{lem:points-in-support-are-limits}.  Using that \(f_i(x_i)=F_i(\phi_i(x_i))\), we get
\begin{equation}\label{eq:triangle-ineq-2}
d(f(x), F_i(\phi(x)))\leq d(f(x), f_i(x_i))+d(F_i(z_i), F_i(z)),
\end{equation}
where \(z_i=\phi_i(x_i)\) and \(z=\phi(x)\).
As \(f\) is the Sormani limit of the \(f_i\)'s, we have that \(f_i(x_i)\) converges to \(f(x)\). Moreover, since \(d(F_i(z_i), F_i(z))\leq L d(z_i, z)\) and \(z_i \to z\) as \(i\to \infty\), it follows from \eqref{eq:triangle-ineq-2} that \(F_i(\phi(x))\) converges to \(f(x)\), as desired. 

Now, since \(F_i\circ \phi\) converges pointwise to \(F\circ \phi\), by using  Definition~\ref{def:current}\eqref{ax:1},\eqref{ax:3} and Lebesgue's dominated convergence theorem, it is easy to check that \(F_{i\#} T\) converges weakly to \(F_\# T\). Since the sequence is uniformly \(\mathbf{N}\)-bounded and $\mathbb{R}^N$ admits coning inequalities for \(\bI_i(\R^N)\) for \(i=1, \dots, n\), it follows that \(F_{i\#}T\) flat converges to \(F_\# T\). Hence, because of \eqref{eq:triangle-ineq-1}, \(f_{i\#} \bb{X_i}\) flat converges to \(f_\# \bb{X}\), as desired.
\end{proof}

\begin{figure}[b]
\centering{
\includegraphics[width=8cm]{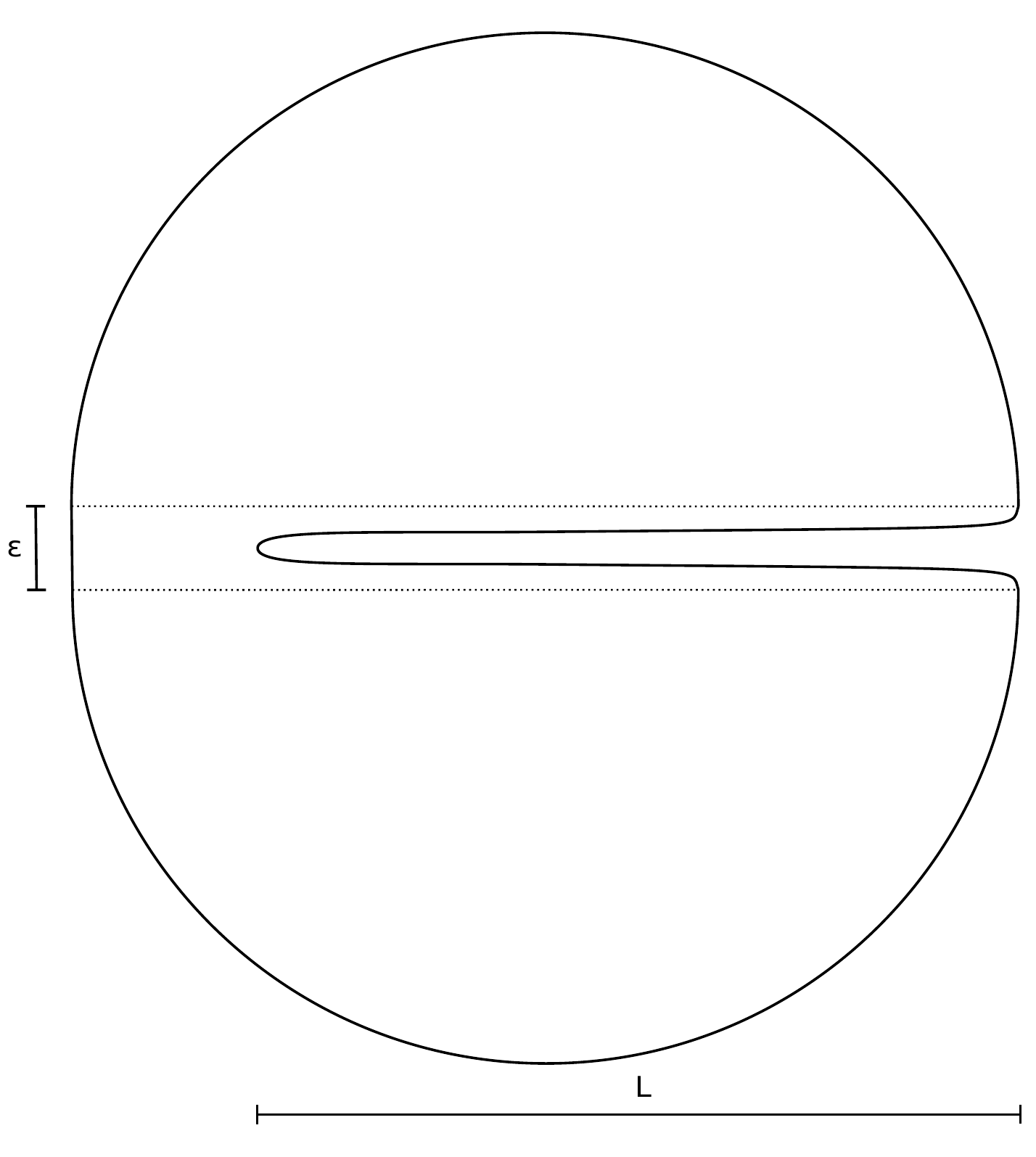}}
\caption{'Flat-football' counterexample to  Question \ref{question:Perales-Sormani}.}
\label{fig:flat-football}
\end{figure}

\subsection{Perales question}\label{sec:7.2}
The following example shows that the Perales question stated in the introduction has a negative answer in general.  The argument uses the following observation, which follows directly from Lemma \ref{lem:points-in-support-are-limits}: If $T_i\in \bI_{n}(Z)$ flat converges to $T\in \bI_n(Z)$ and \(\set T_i\)  Gromov--Hausdorff converges to \(Y\), then \(\spt T\) admits an isometric embedding into \(Y\).

\begin{example}\label{ex:counterexample}
Fix \(L\in (0,2)\) and let \(M_\varepsilon\) denote the flat \(2\)-dimensional Riemannian manifold with boundary depicted in Figure~\ref{fig:flat-football}. Notice that \(M_\varepsilon\) admits a natural decomposition into three pieces, namely \(M_\varepsilon = B_+^2 \cup R_\varepsilon \cup B_-^2\), where \(B_\pm^2\) are isometric to the half-ball \(B^2 \cap \{y \geq 0\}\) and \(R_\varepsilon\) is contained in a rectangle of length \(2\) and width \(\varepsilon\). By construction, \(\vol^2(M_\varepsilon) \to \vol^2(B^2)\) as \(\varepsilon \to 0\). Let \(f_\varepsilon \colon M_\varepsilon \to B^2\) denote the map which collapses \(R_\varepsilon\) to the \(x\)-axis and is the identity on the half-balls \(B_\pm^2\). Clearly, \(f_\varepsilon\) is \(1\)-Lipschitz. Moreover, using that
every cycle in \(\bI_1(S^1)\) is of the form \(m\cdot \bb{S^1}\) for some $m\in \Z$,
it is easy to check that \(f_{\varepsilon \#} \bb{\partial M_\varepsilon} = \bb{S^1}\) for every \(\varepsilon >0\). Let \(U\susbet \R^2\) denote a slit unit disk where the slit has length \(L\). The Gromov-Hausdorff limit of \((M_\varepsilon)\) is equal to the metric completion of \(U\) equipped with the intrinsic metric. In particular, \((M_\varepsilon)\) thus does not converge to \(B^2\) in the intrinsic flat sense.
\end{example}

Question~\ref{question:Perales-Sormani} has a positive answer if, in addition, one assumes a suitable bound for the limit of the masses of the boundary currents. We now prove  Corollary~\ref{prop:answer-to-sormani-perales}, whose statement can be found in the introduction.

\begin{proof}[Proof of Corollary \ref{prop:answer-to-sormani-perales}]
Lemma~\ref{lem:functorality-arzela-ascoli} implies that \(f_{i\#}\bb{X_i}\) flat converges to \(f_\# \bb{X}\). 
Since \(\partial (f_{i\#}\bb{X_i})=f_{i\#}\bb{\partial X_i}\) flat converges to \(\bb{\partial C}\), it follows that \(\partial (f_\# \bb{X})=\bb{\partial C}\). Moreover, by the lower semi-continuity of mass (see \cite[p. 19]{amb00}), we have that
\[
\mass(\bb{X})\leq \liminf _{i\to \infty} \mass(\bb{X_i}) \leq \vol^n(C)
\]
and analogously \(\mass(\bb{\partial X})\le \vol^{n-1}(\partial C)\).  Therefore, by invoking Theorem~\ref{thm:rigidity-result-current-version}, we find that \(f\colon X\to C\) is an isometry. 
\end{proof}

Using this corollary the following result is a direct consequence of the Wenger compactness theorem.

\begin{corollary}\label{cor:flat-convergence}
Let $C\subset \R^n$ be a convex body and \((X_i)\) a sequence of uniformly bounded integral current spaces. Suppose \(f_i \colon X_i \to \R^N\) are \(1\)-Lipschitz maps such that \(f_{i\#}\bb{\partial X_i}\) flat converges to \(\bb{\partial C}\). If 
\[
\lim_{i\to \infty} \mass(\bb{\partial X_i}) \leq \vol^{n-1}(\partial C), \quad \quad \lim_{i\to \infty} \mass(\bb{X_i}) \leq \vol^n(C),
\]
then \((X_i)\) converges in the intrinsic flat sense to \(C\).  
\end{corollary}

\begin{proof}
By \cite[Theorem 1.2]{wenger-2011} there exists a subsequence, also denoted by \(X_i\), that converges in the intrinsic flat sense to an integral current space \(X\). Further, notice that since the \(X_i\) are uniformly bounded, the \(f_i\) take values in a compact set \(K\subset \R^N\). Let \(f\colon X\to K\) denote the Sormani limit of a subsequence of \((f_i)\). The existence of such a limit is guaranteed by Theorem~\ref{thm:Sormani-arzela}. In particular, \(f\) is \(1\)-Lipschitz. Now, Corollary~\ref{prop:answer-to-sormani-perales} tells us that \(f\) is an isometry \(X \to C\). Since the argument above can be applied to any subsequence of \((X_i)\), it follows that \((X_i)\) converges to \(C\) in the intrinsic flat sense, as desired.  
\end{proof}

\section{Counterexamples and open questions}\label{sec:8}
\label{sec:counterexamples-questions}
Theorem~\ref{thm:rigidity-result-current-version} and Corollary~\ref{cor:sphere-rigidity} show that convex bodies in $\mathbb{R}^n$ and the round sphere $S^n$ have the Lipschitz-volume rigidity property among all integral current spaces. This naturally leads to the question which other metric spaces $Y$ are Lipschitz-volume rigid among integral current spaces. A simple way to come up with non-Lipschitz volume rigid spaces is to consider non-intrinsic metrics. In particular, every compact Lipschitz submanifold $Y\subset\mathbb{R}^N$, which is not a convex subset, does not enjoy the Lipschitz-volume rigidity property when it is endowed with the Euclidean subspace metric. In this case the identity map $Y^{\text{int}}\to Y^{\text{euc}}$ is $1$-Lipschitz, volume and boundary volume preserving, but not an isometry. 

Note that in situations where Federer's constancy theorem is not valid the boundary push-forward  condition is not sufficient (e.g.\ for non-trivial spaces $Y$ with $\partial Y=0$).  In the following we refer as \emph{Lipschitz-volume rigidity} of $Y$ to the following property: \emph{Suppose $X$ is an integral current space of the same dimension as~$Y$ and $f\colon X\to Y$ is a $1$-Lipschitz map such that $f_{\#}\bb{X}=\bb{Y}$. If $\mass(\bb{\partial X})\leq \mass(\bb{\partial Y})$ and $\mass(\bb{X})\leq \mass(\bb{Y})$, then $f$ is an isometry.}

\begin{question}
\label{quest:which}
Let $Y\subset \mathbb{R}^N$ be a compact orientable connected $n$-dimensional Lipschitz manifold. Does $Y$ have Lipschitz-volume rigidity among integral current spaces when endowed with its intrinsic metric?
\end{question}

It is not hard to modify the proof of Theorem~\ref{thm:rigidity-result-current-version} to deduce an affirmative answer when $Y$ is smooth and $n=N$, and hence in particular $Y$ is flat. On the other hand Example~4.4 in \cite{CS20} suggests that the answer to Question~\ref{quest:which} is negative for general Lipschitz submanifolds. We suspect that the answer is affirmative when $Y$ is smooth but our proof does not seem amenable for such a generalization in a straightforward way, since it relies on Fubini-type decompositions of $\vol^n(Y)$.

The situation becomes even more complicated when one allows for non infinitesimally Euclidean integral current spaces. It turns out that the Lipschitz-volume rigidity of $Y$ can fail even when $Y$ is a convex body in a finite-dimensional normed space. For example, let $I^2$ be the convex body $[0,1]^2\susbet \R^2$ endowed with the Euclidean metric and $Y=I^2_\infty$ be the same set but endowed with the maximum norm. Then the identity map $f\colon I^2\to Y$ is $1$-Lipschitz, 
\[
\mass(\bb{I^2})=\mu^{m*}(I^2)=1=\mu^{m*}(I^2_\infty)=\mass(\bb{Y})
\]
and 
\[
\mass(\bb{\partial I^2})= \ell(\partial I^2)=4=\ell(\partial I^2_\infty)=\mass(\bb{\partial Y}),
\]
but $f$ is not an isometry.

As discussed in Section~\ref{subsec:jacobians} there is some ambiguity concerning volume as soon as non-Euclidean tangent spaces come into play. The preceeding counterexample stems from the observations that $\mass$ corresponds to the mass\(\ast\) Jacobian $\Jac^{m*}$ in the sense of \eqref{eq:mass*-characterization} and that $\Jac^{m*}(\sigma)$ is not strictly monotone in $\sigma$. Hence, another interesting question would be to investigate whether convex bodies in finite-dimensional normed spaces are Lipschitz-volume rigid among integral current spaces with respect to the Busemann mass $\mass^{\text{b}}$ or the Holmes--Thompson mass~$\mass^{\text{ht}}$.

Concerning Theorem \ref{thm:main-i} we were informed by Roger Züst that Lemma~\ref{lem:absolute-filling-vol} and hence the lower bound \eqref{eq:main-ineq} generalizes to convex bodies in finite-dimensional normed spaces. Indeed for a given Finsler mass $\mass^{\bullet}$ it seems natural to expect that validity of this inequality for all finite-dimensional normed spaces is   equivalent to a property that is often called quasi-convexity or semi-ellipticity over $\mathbb{Z}$ in the literature, see \cite{APT04, Iva08, lw17}. The counterexample above however illustrates that in the setting of normed spaces one can only hope for rigidity when the mass functional is strictly monotone, as is the case for $\mass^{\text{b}}$ or~$\mass^{\text{ht}}$ but not for $\mass$ or~$\mass^{\text{ir}}$.

\bibliographystyle{plain}
\bibliography{abib}

\def\cprime{$'$}
\begin{thebibliography}{10}

\bibitem{allen2020intrinsic}
Brian Allen and Raquel Perales.
\newblock Intrinsic flat stability of manifolds with boundary where volume
  converges and distance is bounded below.
\newblock {\em arXiv preprint arXiv:2006.13030}, 2020.

\bibitem{allen2020volume}
Brian Allen, Raquel Perales, and Christina Sormani.
\newblock Volume above distance below.
\newblock {\em to appear in Journal of Differential Geometry}, 2022.

\bibitem{APT04}
J.~C. \'{A}lvarez Paiva and A.~C. Thompson.
\newblock Volumes on normed and {F}insler spaces.
\newblock In {\em A sampler of {R}iemann-{F}insler geometry}, volume~50 of {\em
  Math. Sci. Res. Inst. Publ.}, pages 1--48. Cambridge Univ. Press, Cambridge,
  2004.

\bibitem{amb00}
Luigi Ambrosio and Bernd Kirchheim.
\newblock {Currents in metric spaces}.
\newblock {\em Acta Math.}, 185(1):1--80, 2000.

\bibitem{Ball92}
Keith Ball.
\newblock Ellipsoids of maximal volume in convex bodies.
\newblock {\em Geom. Dedicata}, 41(2):241--250, 1992.

\bibitem{BCIK05}
V.~Bangert, C.~Croke, S.~Ivanov, and M.~Katz.
\newblock Filling area conjecture and ovalless real hyperelliptic surfaces.
\newblock {\em Geom. Funct. Anal.}, 15(3):577--597, 2005.

\bibitem{basso2021undistorted}
Giuliano Basso, Stefan Wenger, and Robert Young.
\newblock Undistorted fillings in subsets of metric spaces.
\newblock preprint arXiv:2112.11905, 2021.

\bibitem{bcg95}
Gérard Besson, Gilles Courtois, and Sylvestre Gallot.
\newblock Entropies et rigidit\'{e}s des espaces localement sym\'{e}triques de
  courbure strictement n\'{e}gative.
\newblock {\em Geom. Funct. Anal.}, 5(5):731--799, 1995.

\bibitem{Bonicatto-Del-Nin-Pasqualetto-2022}
Paolo Bonicatto, Giacomo Del~Nin, and Enrico Pasqualetto.
\newblock Decomposition of integral metric currents.
\newblock {\em J. Funct. Anal.}, 282(7):Paper No. 109378, 28, 2022.

\bibitem{brudnyi2012}
Alexander Brudnyi and Yuri Brudnyi.
\newblock {\em Methods of geometric analysis in extension and trace problems.
  {V}olume 1}, volume 102 of {\em Monographs in Mathematics}.
\newblock Birkh\"{a}user/Springer Basel AG, Basel, 2012.

\bibitem{burago2022course}
Dmitri Burago, Yuri Burago, and Sergei Ivanov.
\newblock {\em A course in metric geometry}, volume~33 of {\em Graduate Studies
  in Mathematics}.
\newblock American Mathematical Society, Providence, RI, 2001.

\bibitem{bi95}
Dmitri Burago and Sergei Ivanov.
\newblock On asymptotic volume of tori.
\newblock {\em Geom. Funct. Anal.}, 5(5):800--808, 1995.

\bibitem{bi10}
Dmitri Burago and Sergei Ivanov.
\newblock Boundary rigidity and filling volume minimality of metrics close to a
  flat one.
\newblock {\em Ann. of Math. (2)}, 171(2):1183--1211, 2010.

\bibitem{bi13}
Dmitri Burago and Sergei Ivanov.
\newblock Area minimizers and boundary rigidity of almost hyperbolic metrics.
\newblock {\em Duke Math. J.}, 162(7):1205--1248, 2013.

\bibitem{chs22}
Simone Cecchini, Bernhard Hanke, and Thomas Schick.
\newblock Lipschitz rigidity for scalar curvature.
\newblock preprint arXiv:2206.11796, 2022.

\bibitem{cre20}
Paul Creutz.
\newblock Majorization by hemispheres and quadratic isoperimetric constants.
\newblock {\em Trans. Amer. Math. Soc.}, 373(3):1577--1596, 2020.

\bibitem{CS20}
Paul Creutz and Elefterios Soultanis.
\newblock Maximal metric surfaces and the {S}obolev-to-{L}ipschitz property.
\newblock {\em Calc. Var. Partial Differential Equations}, 59(5):Paper No. 177,
  34, 2020.

\bibitem{de-giorgi-1995}
Ennio De~Giorgi.
\newblock General {P}lateau problem and geodesic functionals.
\newblock {\em Atti Sem. Mat. Fis. Univ. Modena}, 43(2):285--292, 1995.

\bibitem{de-Rham-1955}
Georges de~Rham.
\newblock {\em Vari\'{e}t\'{e}s diff\'{e}rentiables. {F}ormes, courants, formes
  harmoniques}.
\newblock Publ. Inst. Math. Univ. Nancago, III. Hermann \& Cie, Paris, 1955.

\bibitem{Upcoming}
Giacomo Del~Nin and Raquel Perales.
\newblock Rigidity of mass-preserving 1-{L}ipschitz maps from integral current
  spaces into {$\mathbb R^n$}.
\newblock {\em arXiv preprint arXiv:2210.06406}, 2022.

\bibitem{EH21}
Behnam Esmayli and Piotr Haj\l~asz.
\newblock The coarea inequality.
\newblock {\em Ann. Fenn. Math.}, 46(2):965--991, 2021.

\bibitem{fed69}
Herbert Federer.
\newblock {\em Geometric measure theory}.
\newblock Die Grundlehren der mathematischen Wissenschaften, Band 153.
  Springer-Verlag New York, Inc., New York, 1969.

\bibitem{fed60}
Herbert Federer and Wendell~H. Fleming.
\newblock {Normal and integral currents}.
\newblock {\em Ann. of Math. (2)}, 72:458--520, 1960.

\bibitem{gro87}
M.~Gromov.
\newblock {Hyperbolic groups}.
\newblock In {\em {Essays in group theory}}, volume~8 of {\em {Math. Sci. Res.
  Inst. Publ.}}, pages 75--263. Springer, New York, 1987.

\bibitem{gromov-1983}
Mikhael Gromov.
\newblock {Filling Riemannian manifolds}.
\newblock {\em Journal of Differential Geometry}, 18(1):1 -- 147, 1983.

\bibitem{HLP20}
Lan-Hsuan Huang, Dan~A. Lee, and Raquel Perales.
\newblock Intrinsic flat convergence of points and applications to stability of
  the positive mass theorem.
\newblock {\em Ann. Henri Poincar\'{e}}, 23(7):2523--2543, 2022.

\bibitem{sormani-2017}
Lan-Hsuan Huang, Dan~A. Lee, and Christina Sormani.
\newblock Intrinsic flat stability of the positive mass theorem for graphical
  hypersurfaces of {E}uclidean space.
\newblock {\em J. Reine Angew. Math.}, 727:269--299, 2017.

\bibitem{sormani-corrigendum-2022}
Lan-Hsuan Huang, Dan~A. Lee, and Christina Sormani.
\newblock Corrigendum to: {I}ntrinsic flat stability of the positive mass
  theorem for graphical hypersurfaces of {E}uclidean space ({J}. {R}eine
  {A}ngew. {M}ath. 727 (2017), 269--299).
\newblock {\em J. Reine Angew. Math.}, 785:273--274, 2022.

\bibitem{Iva08}
S.~V. Ivanov.
\newblock Volumes and areas of {L}ipschitz metrics.
\newblock {\em Algebra i Analiz}, 20(3):74--111, 2008.

\bibitem{kir94}
Bernd Kirchheim.
\newblock {Rectifiable metric spaces: local structure and regularity of the
  {H}ausdorff measure}.
\newblock {\em Proc. Amer. Math. Soc.}, 121(1):113--123, 1994.

\bibitem{lan11}
Urs Lang.
\newblock {Local currents in metric spaces}.
\newblock {\em J. Geom. Anal.}, 21(3):683--742, 2011.

\bibitem{lang-wenger-2011}
Urs Lang and Stefan Wenger.
\newblock The pointed flat compactness theorem for locally integral currents.
\newblock {\em Comm. Anal. Geom.}, 19(1):159--189, 2011.

\bibitem{li15}
Nan Li.
\newblock Lipschitz-volume rigidity in {A}lexandrov geometry.
\newblock {\em Adv. Math.}, 275:114--146, 2015.

\bibitem{li20}
Nan Li.
\newblock Lipschitz-volume rigidity and globalization.
\newblock In {\em Proceedings of the {I}nternational {C}onsortium of {C}hines
  {M}athematicians 2018}, pages 311--322. Int. Press, Boston, MA, 2020.

\bibitem{lw14}
Nan Li and Feng Wang.
\newblock Lipschitz-volume rigidity on limit spaces with {R}icci curvature
  bounded from below.
\newblock {\em Differential Geom. Appl.}, 35:50--55, 2014.

\bibitem{lw17}
Alexander Lytchak and Stefan Wenger.
\newblock Area minimizing discs in metric spaces.
\newblock {\em Arch. Ration. Mech. Anal.}, 223(3):1123--1182, 2017.

\bibitem{par67}
Kalyanapuram~R. Parthasarathy.
\newblock {\em Probability measures on metric spaces}.
\newblock Probability and Mathematical Statistics, No.\ 3. Academic Press,
  Inc., New York-London, 1967.

\bibitem{ps17}
J.~Portegies and C.~Sormani.
\newblock Properties of the intrinsic flat distance.
\newblock {\em Algebra i Analiz}, 29(3):70--143, 2017.

\bibitem{pu52}
P.~M. Pu.
\newblock Some inequalities in certain nonorientable {R}iemannian manifolds.
\newblock {\em Pacific J. Math.}, 2:55--71, 1952.

\bibitem{ruan2022filling}
Yuping Ruan.
\newblock Filling volume minimality and boundary rigidity of metrics close to a
  negatively curved symmetric metric.
\newblock {\em arXiv preprint arXiv:2201.09175}, 2022.

\bibitem{sor18}
Christina Sormani.
\newblock Intrinsic flat {A}rzela-{A}scoli theorems.
\newblock {\em Comm. Anal. Geom.}, 26(6):1317--1373, 2018.

\bibitem{sor22}
Christina Sormani.
\newblock Talk: Integral current spaces and their properties.
\newblock At BIRS Workshop on Integral and Metric Geometry, 2022.
\newblock Video available under
  \url{www.birs.ca/events/2022/5-day-workshops/22w5181/videos}.

\bibitem{sw11}
Christina Sormani and Stefan Wenger.
\newblock The intrinsic flat distance between {R}iemannian manifolds and other
  integral current spaces.
\newblock {\em J. Differential Geom.}, 87(1):117--199, 2011.

\bibitem{wenger07}
Stefan Wenger.
\newblock Flat convergence for integral currents in metric spaces.
\newblock {\em Calc. Var. Partial Differential Equations}, 28(2):139--160,
  2007.

\bibitem{wenger-2011}
Stefan Wenger.
\newblock Compactness for manifolds and integral currents with bounded diameter
  and volume.
\newblock {\em Calc. Var. Partial Differential Equations}, 40(3-4):423--448,
  2011.

\bibitem{williams-2012}
Marshall Williams.
\newblock Metric currents, differentiable structures, and {C}arnot groups.
\newblock {\em Ann. Sc. Norm. Super. Pisa Cl. Sci. (5)}, 11(2):259--302, 2012.

\bibitem{zust-2019}
Roger Z\"{u}st.
\newblock Functions of bounded fractional variation and fractal currents.
\newblock {\em Geom. Funct. Anal.}, 29(4):1235--1294, 2019.

\bibitem{zust2021riemannian}
Roger Z{\"u}st.
\newblock The {R}iemannian hemisphere is almost calibrated in the injective
  hull of its boundary.
\newblock {\em arXiv preprint arXiv:2104.04498}, 2021.

\end{thebibliography}

\end{document}